\documentclass[12pt,a4paper]{amsart}

\usepackage[T1]{fontenc}
\usepackage{a4wide}

\usepackage{amsmath,amssymb}
\usepackage{amsthm}
\usepackage{dsfont}
\usepackage{bbm}
\usepackage{thmtools}
\usepackage{thm-restate}

\usepackage{graphicx}
\usepackage{tikz}
\usepackage{pgfplots}
\usepackage{float}
\usepackage{subcaption}
\usepackage{placeins}

\pgfplotsset{compat=1.18}

\usepackage{xcolor}
\usepackage[normalem]{ulem}
\usepackage{xargs}
\usepackage{comment}
\usepackage{todonotes} % provides \todo

\usepackage[hidelinks]{hyperref}
\usepackage{cleveref}

\newif\ifthesis
\thesisfalse
\newif\ifdetails
\detailsfalse
\newtheorem{theorem}{Theorem}[section]

\newtheorem{lemma}[theorem]{Lemma}
\newtheorem{proposition}[theorem]{Proposition}
\newtheorem{corollary}[theorem]{Corollary}

\theoremstyle{definition}
\newtheorem{definition}[theorem]{Definition}

\theoremstyle{remark}
\newtheorem{remark}[theorem]{Remark}
\newtheorem{remarks}[theorem]{Remarks}

\crefformat{theorem}{#2#1#3}
\crefformat{fact}{#2#1#3}
\crefformat{lemma}{#2#1#3}
\crefformat{proposition}{#2#1#3}
\crefformat{corollary}{#2#1#3}
\crefformat{result}{#2#1#3}
\crefformat{assumption}{(#2#1#3)}
\crefformat{definition}{(#2#1#3)}
\crefformat{remark}{#2#1#3}
\crefformat{example}{#2#1#3}

\crefformat{equation}{(#2#1#3)}
\crefformat{notation}{(#2#1#3)}
\crefformat{section}{#2#1#3}
\crefformat{item}{#2#1#3}
\crefformat{figure}{#2#1#3}

\newcommand{\N}{\mathbb{N}}
\newcommand{\R}{\mathbb{R}}

\newcommand{\Z}{\mathbb{Z}}

\newcommand{\E}{\mathcal{E}}

\title{Critical Discrete Hardy Inequalities in $L^p$}
\author[K.~Bogdan]{Krzysztof Bogdan}
\address{Faculty of Pure and Applied Mathematics, Wroc\l{}aw University of Science and Technology, Wyb. Wyspia'nskiego 27, 50-370 Wroc\l{}aw, Poland.}
\email{[krzysztof.bogdan@pwr.edu.pl](mailto:krzysztof.bogdan@pwr.edu.pl)}

\author[S.~Gupta]{Shubham Gupta}
\address{Department of Mathematics, Technion--Israel Institute of Technology, Haifa, Israel.}
\email{[sg1019@campus.technion.ac.il](mailto:sg1019@campus.technion.ac.il)}

\author[K.~Kaleta]{Kamil Kaleta}
\address{Faculty of Pure and Applied Mathematics, Wroc\l{}aw University of Science and Technology, Wyb. Wyspia'nskiego 27, 50-370 Wroc\l{}aw, Poland.}
\email{[kamil.kaleta@pwr.edu.pl](mailto:kamil.kaleta@pwr.edu.pl)}

\author[A.~Szczukiewicz]{Antoni Szczukiewicz}
\address{Faculty of Pure and Applied Mathematics, Wroc\l{}aw University of Science and Technology, Wyb. Wyspia'nskiego 27, 50-370 Wroc\l{}aw, Poland.}
\email{264032@student.pwr.edu.pl}

\thanks{The research was supported in part by the OPUS grants 2023/51/B/ST1/02209 and 2023/49/B/ST1/00678 of the National Science Centre (Poland), by the ISF-DFG Lead Agency research grant (ISF No. 1531/25), and by Wroc{\l}aw University of Science and Technology through the Academia Professorum Iuniorum programme.}

\date{August 31, 2026}


\begin{document}

\begin{abstract}
We establish ground-state representations and critical Hardy inequalities in the discrete $L^p$ setting. In particular, we construct critical Hardy weights for 
the discrete Dirichlet Laplacian on the half-line and the discrete fractional Laplacian on the integers
for all $p\in (1,\infty)$. Our approach uses Sobolev--Bregman forms, for which the ground-state representations are exact identities and the corresponding Hardy weights are expressed through linear operators acting on powers of positive superharmonic functions.
\end{abstract}

\maketitle

\section{Introduction}
Hardy-type inequalities play an important role in spectral theory, mathematical physics, partial differential equations, potential theory, and probability. The subject goes back to the early twentieth century;  for historical background, see Kufner, Maligranda, and Persson \cite{MR2256532}. In what follows, we consider $$p\in(1,\infty).$$ Let $\N:=\{1,2,\ldots\}$ and $\N_0:=\{0,1,2,\ldots\}$.
The classical discrete Hardy inequality reads
\begin{equation}\label{classical lp Hardy}
\sum_{x \in \N} \bigl|u(x)-u(x-1)\bigr|^p
\geq
\left(\frac{p-1}{p}\right)^p
\sum_{x \in \N} \frac{|u(x)|^p}{x^p}
\end{equation}
for finitely supported functions $u:\N_0\to\R$ satisfying $u(0)=0$.
For $p=2$, we get
\begin{equation}\label{classical quadratic Hardy}
\sum_{x \in \N} \bigl(u(x)-u(x-1)\bigr)^2
\geq
\frac{1}{4} \sum_{x \in \N} \frac{u(x)^2}{x^2}.
\end{equation}
We consider the non-negative discrete Dirichlet Laplacian on the
half-line,
$$
\Delta u(x)
:=
2u(x)-u(x-1)-u(x+1),
\qquad x\in\N,
$$
where $u(0)=0$. By the identity \eqref{e.rg}, the inequality \eqref{classical quadratic Hardy}
can equivalently be written as
$$
\langle\Delta u,u\rangle_\N
\geq
\langle wu,u\rangle_\N,
$$
where
$$
w(x):=\frac{1}{4x^2},
\qquad x\in\N,
$$
and
$$
\langle f,g\rangle_\N
:=
\sum_{x\in\N}f(x)g(x).
$$
%is the usual inner product on $\ell^2(\N)$, the space of square-summable functions on $\N$.
The constant $1/4$ in \eqref{classical quadratic Hardy} is sharp: it cannot be replaced by a larger one. Keller, Pinchover, and Pogorzelski \cite{MR3779222} showed, however, that the weight $1/(4x^2)$ can be improved pointwise. More precisely, they replaced \eqref{classical quadratic Hardy} by
\begin{equation}\label{KPP Hardy}
\sum_{x \in \N} \bigl(u(x)-u(x-1)\bigr)^2
\geq
\sum_{x \in \N} w^{\mathrm{KPP}}(x)u(x)^2,
\end{equation}
where
\begin{equation}\label{KPP weight}
w^{\mathrm{KPP}}(x)
:=
2-\sqrt{1-\frac1x}-\sqrt{1+\frac1x}>\frac{1}{4x^2},
\qquad x\in\N.
\end{equation}
\ifdetails
Indeed, let \(s:=1/n\) and put
\[
    S:=\sqrt{1-s}+\sqrt{1+s}.
\]
Then
\[
    w^{\mathrm{KPP}}(n)
    =
    2-S.
\]
Since
\[
    S^2
    =
    2+2\sqrt{1-s^2},
\]
we get
\[
\begin{aligned}
    w^{\mathrm{KPP}}(n)
    &=
    2-S
    =
    \frac{4-S^2}{2+S}  \\
    &=
    \frac{2-2\sqrt{1-s^2}}{2+S}  \\
    &=
    \frac{2s^2}{\bigl(1+\sqrt{1-s^2}\bigr)(2+S)}.
\end{aligned}
\]
Now \(0<s\leq1\), hence
\[
    1+\sqrt{1-s^2}\leq2
\]
and
\[
    S=\sqrt{1-s}+\sqrt{1+s}<2.
\]
Therefore
\[
    \bigl(1+\sqrt{1-s^2}\bigr)(2+S)<8.
\]
It follows that
\[
    w^{\mathrm{KPP}}(n)
    >
    \frac{2s^2}{8}
    =
    \frac{s^2}{4}
    =
    \frac{1}{4n^2}.
\]
\fi
Furthermore, they proved that the weight $w^{\mathrm{KPP}}$ satisfies strong optimality properties. In particular, it is critical: it cannot be increased even at a single point without destroying the validity of \eqref{KPP Hardy}. Critical Hardy weights have been studied extensively in recent years, especially in the setting of discrete graphs \cite{MR5007873,MR5063900,PhilippHake,MR3774437,StampachWaclawek2026}. We also refer to \cite{MR3170212,MR3436428,hou2025optimal} for related results in the continuum. In the $\ell^p$ setting, analogous questions have been considered on locally finite graphs \cite{MR4768506,vstampach2026optimal}. Much less is known, however, for nonlocal discrete operators whose kernels couple each point to infinitely many others. Here a basic example is the discrete fractional Laplacian on $\Z$, whose kernel couples every pair of distinct points.
We thus seek analogues of \eqref{KPP Hardy} for general $p\in(1,\infty)$.
The left-hand side of \eqref{classical lp Hardy},
\begin{equation}\label{p energy}
\sum_{x \in \N}|u(x)-u(x-1)|^p,
\end{equation}
is the usual Sobolev form and a natural candidate.
In the present paper, however, we choose to work with a different quantity. Namely we argue for the use in Hardy-type inequalities of the following \emph{Sobolev--Bregman form}
\begin{equation}\label{def: SB}
\mathcal{E}_{p}^{X,k}(u)
:=
\frac{1}{2}
\sum_{x\in X}\sum_{y\in X}
(u(x)-u(y))
\bigl(u(x)^{\langle p-1\rangle}
-u(y)^{\langle p-1\rangle}\bigr)
k(x,y).
\end{equation}
Here $X$ is a countable set, $k:X\times X\to[0,\infty)$ is a nonnegative symmetric function called \emph{kernel}, $u:X\to\R$, and, for $\beta>0$, we denote
$$
a^{\langle \beta\rangle}:=|a|^\beta\operatorname{sgn}a,
\qquad a\in\R.
$$
The Sobolev--Bregman form is associated with the convex function
$$
\Phi_p(t):=\frac{|t|^p}{p},\qquad t\in\R,
$$
and couples the increments of $u$ with the increments of
\begin{equation}\label{e.dpf}
\Phi_p'(a)=a^{\langle p-1\rangle}, \qquad a\in\R.
\end{equation}
Since $\Phi_p'$ is increasing, each summand in \eqref{def: SB} is
nonnegative, so the sum is well defined, although possibly infinite.
Precise assumptions on $k$ and a convenient class of functions $u$
making $\mathcal{E}_{p}^{X,k}(u)$ finite will be specified later. Our main abstract result, Theorem~\ref{thm: abstract ground state representation},
establishes the Hardy inequality \eqref{abstract Hardy inequality} and
gives an exact ground-state representation of its deficit. More
precisely, it constructs a Hardy weight $w_h$ from a suitable positive
superharmonic function $h$ and expresses
$$
\mathcal E_p^{X,k}(u)
-
\sum_{x\in\Omega}w_h(x)|u(x)|^p
$$
as the explicit nonnegative remainder in
\eqref{abstract ground state representation}.

To further motivate the result and the form, which might look unfamiliar at first, we consider the discrete Dirichlet Laplacian on the half-line. It is induced by the nearest-neighbour kernel
$$
k(x,y):=
\begin{cases}
1, & \text{if } |x-y|=1,\\
0, & \text{otherwise}.
\end{cases}
$$
For any function $u:\N_0\to\R$ satisfying $u(0)=0$, the form
\eqref{def: SB} becomes
\begin{align}\label{Dirichlet Laplacian p-form}
\mathcal{E}_{p}^{\N_0}(u)
&:=
\frac12
\sum_{\substack{x,y\in\N_0\\ |x-y|=1}}
\bigl(u(x)-u(y)\bigr)
\bigl(
u(x)^{\langle p-1\rangle}
-
u(y)^{\langle p-1\rangle}
\bigr) \nonumber\\
&=
\sum_{x\in\N}
\bigl(u(x)-u(x-1)\bigr)
\bigl(
u(x)^{\langle p-1\rangle}
-
u(x-1)^{\langle p-1\rangle}
\bigr).
\end{align}
If, in addition, $u\in\ell^p(\N)$, then all the series occurring
in the following computation are absolutely convergent by Hölder's inequality: 
\begin{align}\label{e.rg}
    \mathcal{E}_{p}^{\N_0}(u)
    &=
        \sum_{x \in \N}
    \bigl(u(x)-u(x-1)\bigr)u(x)^{\langle p-1\rangle}
    +
    \sum_{x \in \N}
    \bigl(u(x)-u(x+1)\bigr)u(x)^{\langle p-1\rangle}
    \nonumber\\
    &=
    \bigl\langle \Delta u,u^{\langle p-1\rangle}\bigr\rangle_\N.
\end{align}
Thus, even though $\mathcal{E}_{p}^{\N_0}$ is no longer a quadratic
form when $p\neq2$, it is directly linked to the discrete Laplacian
through the pairing of $\Delta u$ with $u^{\langle p-1\rangle}$. Such forms in the continuum can be traced back to \cite{stroock2012introduction} and arise naturally in the study of semigroups;  see the discussion in
Subsection~\ref{sec:c}.
We emphasize again that the Sobolev--Bregman form \eqref{def: SB} is fundamentally different from the Sobolev form \eqref{p energy}.
When $p=2$, both the Sobolev and the Sobolev--Bregman
forms reduce to the same quadratic Dirichlet form, to wit, $\mathcal{E}_{2}^{X,k}(u)$. For $p\neq2$,
however, they provide two different extensions of the
quadratic theory. The former is based on the $p$-th power of the increments, whereas
the latter retains the operator-theoretic structure
$\langle \Delta u,u^{\langle p-1\rangle}\rangle$.

Ground-state representations for Schrödinger operators in the continuum
are classical tools in the spectral theory of differential operators.
Their analogues in discrete and nonlocal settings have proved particularly
effective in constructing Hardy weights with various optimality properties
\cite{MR4597627,MR2415477,MR4680228,vstampach2026optimal,
StampachWaclawek2026}. The purpose of this paper is to develop a discrete
$L^p$ theory of critical Hardy inequalities, with Sobolev--Bregman forms
providing the underlying nonlinear framework. In particular, we seek
nonquadratic analogues of the criticality theory for Hardy weights on
graphs developed by Keller, Pinchover, and Pogorzelski
\cite{MR3774437}.

The use of Sobolev--Bregman forms is rooted in several earlier developments.
Hardy--Stein identities and related decompositions of $L^p$ norms in terms
of Bregman divergences were established in
\cite{MR3251822,MR3556449}, while the relation between $L^p$ Hardy
inequalities for the fractional Laplacian and contractivity of the
corresponding perturbed semigroups was exploited in \cite{MR4372148}.
A general structural foundation was provided by the Beurling--Deny formula
for Sobolev--Bregman forms obtained by Gutowski and Kwaśnicki
\cite{MR4885983}, which identifies these forms as natural $L^p$
counterparts of Dirichlet forms.

Further motivation comes from fractional Hardy inequalities on the
integers and the discrete half-line \cite{MR4612316,DasFuenteFernandez2026},
Douglas-type identities in $L^p$ \cite{MR4600287,MR4589708}, and, most
directly, critical fractional Hardy inequalities for Sobolev--Bregman
forms associated with homogeneous kernels on $\R$
\cite{BogdanDydaSzczukiewicz2026}. Together, these developments also
suggest the use of Hardy inequalities for Sobolev--Bregman forms in the
study of $L^p$-contractivity of semigroups; we present this application
in Subsection~\ref{sec:c}.

From the point of view of Hardy inequalities, Sobolev--Bregman forms
are in many respects simpler and more transparent than the usual
Sobolev $p$-forms \eqref{p energy}. Their ground-state representations
are exact identities, whereas for Sobolev $p$-forms one obtains
two-sided inequalities known as the \emph{simplified energy estimates}
\cite{MR4597627}. Moreover, the resulting Hardy weights are expressed
in terms of linear operators, with nonlinear dependence appearing only
through powers of superharmonic functions. For Sobolev $p$-forms, by
contrast, the operator itself is nonlinear, which makes the analysis
of Hardy weights more involved.

We establish ground-state representations for the Sobolev--Bregman
forms \eqref{def: SB} on arbitrary discrete spaces $X$ and use them to
derive Hardy weights. We then prove the criticality of the resulting
weights in two concrete settings. The Dirichlet Laplacian on $\N_0$
provides a transparent illustration of the general method, with an
additional feature concerning the class of admissible functions.
The fractional Laplacian on $\Z$ provides a genuinely nonlocal
setting: $k(x,y)>0$ for all distinct $x,y\in\Z$, so interactions
between arbitrarily distant points must be taken into account.
Throughout the paper, we restrict ourselves to discrete spaces
equipped with the counting measure; a systematic extension to general
reference measures will be pursued elsewhere.

We now state our results for these two model cases, beginning with the
Dirichlet Laplacian on $\N_0$.
\subsection{Dirichlet Laplacian}\label{subsec: Dirichlet Laplacian}

We begin with the simpler case of the Dirichlet form $\mathcal{E}_p^{\N_0}$ given by \eqref{Dirichlet Laplacian p-form}. The quadratic case $p=2$ is well studied, and since the discovery of the critical weight $w^{\mathrm{KPP}}$ in \eqref{KPP weight}, infinitely many other critical Hardy weights have been found. In particular, Krej\v{c}i\v{r}\'{\i}k, Laptev, and \v{S}tampach \cite{MR4549127} proved that the weights
$$
w_\alpha(x)
:=
2-
\left(1-\frac{1}{x}\right)^\alpha
-
\left(1+\frac{1}{x}\right)^\alpha,
\qquad x\in\N,
$$
are critical Hardy weights for the form $\mathcal{E}_2^{\N_0}$ for $\alpha\in(0,1/2]$. We extend the result of \cite{MR4549127} to all $p>1$. Below, as usual, $p'\in (1,\infty)$ is the conjugate exponent of $p$, given by
\begin{equation}\label{conjugate exponent}
p':=\frac{p}{p-1},
\qquad
\frac{1}{p}+\frac{1}{p'}=1,
\qquad
(p-1)(p'-1)=1.
\end{equation}
Throughout, $a\wedge b:=\min\{a,b\}$ and $a\vee b:=\max\{a,b\}$ and
we denote
\begin{equation}\label{a_p constant}
a_p:=(p-1)\wedge (p'-1)
=
\begin{cases}
p-1, & 1<p\leq 2,\\[1mm]
p'-1, & 2\leq p<\infty.
\end{cases}
\end{equation}

The following theorem gives a Hardy-type inequality for the Sobolev--Bregman form associated with the discrete Dirichlet Laplacian $\Delta$.

\begin{theorem}\label{thm: critical Hardy weights for Dirichlet Laplacian}
If $p>1$, $\alpha\in(0,1]$, $(p,\alpha)\neq(2,1)$, $u(0)=0$, and $u\in\ell^p(\N,x^{-2})$, then
\begin{equation}\label{e.HiSB}
    \mathcal{E}_p^{\N_0}(u)
    \geq
    \sum_{x\in\N} w_{p,\alpha}(x)|u(x)|^p,
\end{equation}
where
\begin{equation}\label{e.Dirichlet Hardy weight}
    w_{p,\alpha}(x)
    :=
    \frac{1}{a_p+1}
    \frac{\Delta x^{\alpha a_p}}{x^{\alpha a_p}}
    +
    \frac{a_p}{a_p+1}
    \frac{\Delta x^\alpha}{x^\alpha},
    \qquad x\in\N.
\end{equation}
The Hardy weight $w_{p,\alpha}$ is strictly positive; it is critical whenever
\begin{equation}\label{e.Dirichlet critical range}
\alpha\leq\frac{1}{a_p+1}.
\end{equation}
\end{theorem}
We defer the proof of Theorem~\ref{thm: critical Hardy weights for Dirichlet Laplacian} to Section~\ref{sec: Dirichlet Laplacian}.

\begin{remark}
 We note that the condition $u(0) = 0$ is necessary. For $p=2$, it is well-known that the form \eqref{Dirichlet Laplacian p-form} without this boundary condition is not subcritical, that is, it does not admit a nonzero Hardy weight \cite{MR429378}. By \eqref{equivalence of 2 and p form}, the same conclusion extends to all $p > 1$.    
\end{remark}

\begin{remark}
For $p=2$, Theorem~\ref{thm: critical Hardy weights for Dirichlet Laplacian}
recovers the result of \cite{MR4549127}. By the binomial expansion, 
\begin{equation}\label{Taylor expansion of Dirichlet Hardy weight}
w_{p,\alpha}(x)
=
-2\sum_{i=1}^{\infty}
\left(
\frac{1}{a_p+1}\binom{\alpha a_p}{2i}
+
\frac{a_p}{a_p+1}\binom{\alpha}{2i}
\right)x^{-2i},\qquad x\geq2.
\end{equation}
Since $\alpha,\alpha a_p\in(0,1]$, all terms in this series are
nonnegative. Consequently,
\begin{equation}\label{e.lb}
w_{p,\alpha}(x)\geq c_{p,\alpha}x^{-2},
\qquad x\geq2,
\end{equation}
where
\begin{align*}
c_{p,\alpha}
&:=
-2\left(
\frac{1}{a_p+1}\binom{\alpha a_p}{2}
+
\frac{a_p}{a_p+1}\binom{\alpha}{2}
\right) 
=
\frac{\alpha a_p}{a_p+1}
\bigl(2-\alpha(a_p+1)\bigr).
\end{align*}
The constant $c_{p,\alpha}$ is maximal for
$\alpha=1/(a_p+1)$, in which case
\[
c_{p,\alpha}
=
\frac{a_p}{(a_p+1)^2}
=
\frac{p-1}{p^2}
=\frac{1}{pp'}.
\]
Together with a separate estimate at $x=1$, this leads to the sharp
inverse-square Hardy inequality stated below.
\end{remark}

\begin{theorem}\label{thm: sharp inverse-square Hardy}
Let $p>1$. If $u(0)=0$ and $u\in\ell^p(\N,x^{-2})$, then
\begin{equation}\label{e.sharp inverse-square Hardy}
\mathcal{E}_p^{\N_0}(u)
\geq
\left(\frac{p-1}{p^2}\right)
\sum_{x\in\N}\frac{|u(x)|^p}{x^2}.
\end{equation}
Moreover, the constant $(p-1)/p^2$ is sharp.
%, that is, \eqref{e.sharp inverse-square Hardy} fails with any larger constant.
\end{theorem}

The proof is given in Section~\ref{sec: Dirichlet Laplacian}.

\begin{remark}
For
\[
\alpha=\frac{1}{a_p+1},
\]
the leading term of $w_{p,\alpha}$ is the classical inverse-square
weight:
\[
w_{p,\alpha}(x)
=
\frac{p-1}{p^2}\frac{1}{x^2}
+
O(x^{-4})
\qquad\text{as }x\to\infty.
\]
In fact,
\[
w_{p,1/(a_p+1)}(x)
>
\frac{p-1}{p^2}\frac{1}{x^2},
\qquad x\in\N.
\]
By Theorem~\ref{thm: critical Hardy weights for Dirichlet Laplacian},
the weight on the left-hand side is critical. Hence, the classical
inverse-square weight is not critical and admits a pointwise improvement
to a critical weight. This extends \eqref{KPP Hardy} from $p=2$ to all
$p>1$.

\end{remark}

\begin{comment}
\begin{remark}[The natural weighted domain]
The appearance of the weight $x^{-1}$ in the $\ell^p$ space comes from the
abstract ground-state representation, see \eqref{abstract reference measure} for definition of $\mu_h$. Indeed, for
\[
g(x):=x^\alpha,
\]
the corresponding reference weight satisfies
\[
\nu_{g}(x)\asymp x^{-1},
\]
whereas the resulting Hardy weight has the asymptotic behaviour
\[
w_{p,\alpha}(x)
  =c_{p,\alpha}x^{-2}+O(x^{-4}),
  \qquad x\to\infty,
\]
with $c_{p,\alpha}>0$ for $0<\alpha<1$. Consequently,
\[
\ell^p(\N)
 \subsetneq \ell^p(\N,x^{-1})
 \subsetneq \ell^p(\N,w_{p,\alpha})
 =\ell^p(\N,x^{-2}),
\]
In the present model of Dirichlet Laplacian on half-line, however, the corresponding Hardy inequality can be extended to every
\[
u\in \ell^p(\N,w_{p,\alpha}).
\]
This extension does not follow directly from the abstract ground-state representation and requires an additional truncation argument.

In Section \ref{sec: Hardy weight reference measure}, we study this issue separately for the Dirichlet Laplacian on the half-line and for the fractional Laplacian on $\Z$. In both cases, suitable truncations allow us to pass from the reference weighted space supplied by the abstract theory to the full natural space determined by the Hardy weight. These examples suggest that such an extension principle may hold in a more general class of discrete models, although it is false in full generality without any additional assumptions; see Remark \ref{rem: reference measure}.
\end{remark}
\end{comment}

\subsection{Fractional Laplacian}\label{subsec: fractional Laplacian}

Next, we consider the case where $X=\Z$ and $k$ is the kernel of the fractional Laplacian. Since $\Delta$ is a bounded nonnegative self-adjoint operator on $\ell^2(\Z)$, for $\sigma\in(0,1)$ its fractional power is defined by the spectral theorem as
\begin{equation}
\Delta^\sigma u(x)
:=
\frac{1}{|\Gamma(-\sigma)|}
\int_0^\infty
\bigl(u(x)-e^{-t\Delta}u(x)\bigr)
\frac{dt}{t^{1+\sigma}}.
\end{equation}
The fractional Laplacian admits the following representation as a nonlocal graph Laplacian:
\begin{equation}\label{graph Laplacian representation of fractional Laplacian}
\Delta^\sigma u(x)
=
\sum_{y\in\Z}\bigl(u(x)-u(y)\bigr)k_\sigma(x-y),
\qquad u\in\ell^2(\Z).
\end{equation}
Here $k_\sigma$ is a non-negative kernel such that $k_\sigma(0)=0$, while, for $x\in\Z\setminus\{0\}$,
$$
k_\sigma(x)
:=
\mathcal{A}_{-\sigma}
\frac{\Gamma(|x|-\sigma)}
{\Gamma(|x|+1+\sigma)},
$$
where
\begin{equation}\label{e.leading-constant-ksigma}
\mathcal{A}_{-\sigma}
:=
\frac{4^\sigma\Gamma(1/2+\sigma)}
{\sqrt{\pi}|\Gamma(-\sigma)|};
\end{equation}
see \cite{MR3882021,MR4612316}.
We note that $k_\sigma$ is not finitely supported, making $\Delta^\sigma$ a nonlocal operator, in stark contrast to the local operator $\Delta$. Moreover, by the standard asymptotics for ratios of gamma functions, as $|z|\to\infty$,
\[
k_\sigma(z)
=
\mathcal{A}_{-\sigma}\frac{1}{|z|^{1+2\sigma}}
+
O\left(\frac{1}{|z|^{2+2\sigma}}\right).
\]
Thus, the kernel has the same order of decay as the kernel of the one-dimensional fractional Laplacian of order $2\sigma$ on $\R$. Using the representation above and the symmetry of $k_\sigma$, we obtain, for $u\in C_c(\Z)$,
\begin{equation}\label{def: 2-form}
\mathcal{E}_2^{\Z,\sigma}(u)
:=
\langle\Delta^\sigma u,u\rangle_{\ell^2(\Z)}
=
\frac{1}{2}
\sum_{x\in\Z}\sum_{y\in\Z}
\bigl(u(x)-u(y)\bigr)^2k_\sigma(x-y).
\end{equation}
This identity may also be viewed as an instance of Green's formula on graphs; see \cite{MR4680228}. The form $\mathcal{E}_2^{\Z,\sigma}$ is the quadratic member of the Sobolev--Bregman family associated with $\Delta^\sigma$.

Ciaurri and Roncal \cite{MR3882021} proved that, for
$0<\sigma<\alpha<1/2$ and $u\in C_c(\Z)$, the following
ground-state representation holds:
\begin{equation}\label{ground state of 2 form}
\mathcal{E}_2^{\Z,\sigma}(u)
=
\sum_{x\in\Z}w_{\sigma,\alpha}(x)|u(x)|^2
+
\frac{1}{2}
\sum_{x\in\Z}\sum_{y\in\Z}
\left(
\frac{u(x)}{k_{-\alpha}(x)}
-
\frac{u(y)}{k_{-\alpha}(y)}
\right)^2
k_{-\alpha}(x)k_{-\alpha}(y)k_\sigma(x-y),
\end{equation}
where the positive weight $w_{\sigma,\alpha}$ is given by
$$
w_{\sigma,\alpha}(x)
:=
\frac{k_{\sigma-\alpha}(x)}{k_{-\alpha}(x)}
=
\frac{c_{\sigma,\alpha}}{|x|^{2\sigma}}
+
O\left(\frac{1}{|x|^{1+2\sigma}}\right)
\qquad\text{as }|x|\to\infty,
$$
with
$$
c_{\sigma,\alpha}
:=
4^\sigma
\frac{
\Gamma(1/2-\alpha+\sigma)\Gamma(\alpha)
}{
\Gamma(1/2-\alpha)\Gamma(\alpha-\sigma)
}.
$$
Dropping the nonnegative remainder term in \eqref{ground state of 2 form} yields a one-parameter family of Hardy inequalities for the form $\mathcal{E}_2^{\Z,\sigma}(u)$, with weights $w_{\sigma,\alpha}$. The constant $c_{\sigma,\alpha}$ is maximized at $\alpha=(1+2\sigma)/4$, and the leading term of $w_{\sigma,(1+2\sigma)/4}$ coincides with the classical Hardy weight for the fractional Laplacian on $\R$ \cite{herbst1977spectral}.

The question of whether any of the weights $w_{\sigma,\alpha}$ are critical was left open in \cite{MR3882021}. It was later proved in \cite{MR4612316} that, for $0<\sigma<\alpha<1/2$, the weight $w_{\sigma,\alpha}$ is critical if and only if
$$
\alpha\leq\frac{1+2\sigma}{4}.
$$
These results were subsequently extended first to the higher-dimensional lattice $\Z^d$ \cite{hake2025optimal} and then to arbitrary infinite graphs \cite{hake2026positive, hake2026supersolution}, yielding explicit critical Hardy weights. For an analogous result for the fractional Laplacian on the discrete half-line, see \cite{DasFuenteFernandez2026, stampach2026optimalfractionaldiscretehardy}. Related optimal Hardy--Rellich--Birman inequalities for powers of the discrete Laplacian were obtained in \cite{StampachWaclawek2026}.

We extend these results to the fractional $\ell^p$ Sobolev--Bregman form
\begin{equation}\label{def: fractional SB form}
\mathcal{E}_p^{\Z,\sigma}(u)
:=
\frac{1}{2}
\sum_{x\in\Z}\sum_{y\in\Z}
\bigl(u(x)-u(y)\bigr)
\bigl(
u(x)^{\langle p-1\rangle}
-
u(y)^{\langle p-1\rangle}
\bigr)
k_\sigma(x-y),
\end{equation}
for $p>1$, and obtain corresponding critical Hardy weights.
\begin{theorem}\label{thm: critical Hardy weights for fractional Laplacian}
Let $p>1$, $\sigma\in(0,1/2)$, and $\alpha\in(\sigma,1/2)$. If
$u\in\ell^p(\Z,(1+|x|)^{-2\sigma})$, then
\begin{equation}\label{fractional p Hardy inequality}
\mathcal{E}_p^{\Z,\sigma}(u)
\geq
\sum_{x\in\Z}w_{p,\sigma,\alpha}(x)|u(x)|^p,
\end{equation}
where $w_{p,\sigma,\alpha}$ is given by
\begin{equation}\label{fractional p Hardy weight}
w_{p,\sigma,\alpha}(x)
:=
\frac{1}{a_p+1}
\frac{\Delta^\sigma k_{-\alpha}^{a_p}(x)}
     {k_{-\alpha}^{a_p}(x)}
+
\frac{a_p}{a_p+1}
\frac{\Delta^\sigma k_{-\alpha}(x)}
     {k_{-\alpha}(x)},
\qquad x\in\Z.
\end{equation}
The Hardy weight $w_{p, \sigma, \alpha}$ is strictly positive and as $|x|\to\infty$,
\begin{equation}\label{fractional p Hardy weight asymptotics}
w_{p,\sigma,\alpha}(x)
=
\frac{\Psi(p,\sigma,\alpha)}{|x|^{2\sigma}}
+
O\left(
\frac{1}{
|x|^{2\sigma+1-\max\{2\sigma,1-2\alpha\}}
}
\right),
\end{equation}
where
\begin{equation}\label{fractional p Hardy leading constant}
\Psi(p,\sigma,\alpha)
:=
\frac{1}{a_p+1}
\frac{\mathcal{A}_{\gamma_p-\sigma}}
     {\mathcal{A}_{\gamma_p}}
+
\frac{a_p}{a_p+1}
\frac{\mathcal{A}_{\alpha-\sigma}}
     {\mathcal{A}_{\alpha}},
\end{equation}
with
\begin{equation}\label{fractional gamma p}
\gamma_p
:=
\frac{1-a_p(1-2\alpha)}{2}.
\end{equation}
Moreover, $w_{p,\sigma,\alpha}$ is critical whenever
\begin{equation}\label{fractional Hardy critical range}
\alpha
\leq
\frac{2\sigma+a_p}{2(a_p+1)}.
\end{equation}
\end{theorem}

\begin{remarks}
\begin{enumerate}
\item[(a)] For $p=2$, we recover the critical Hardy weights for
$\mathcal{E}_2^{\Z,\sigma}(u)$ from \cite{MR4612316}.

\item[(b)] For fixed $p$ and $\sigma$, all the Hardy weights
$w_{p,\sigma,\alpha}$ have the same order $|x|^{-2\sigma}$ at
infinity. Their leading asymptotic coefficient
$\Psi(p,\sigma,\alpha)$ is maximized at
\[
\alpha
=
\frac{2\sigma+a_p}{2(a_p+1)};
\]
see \cite[Theorem~2]{MR4372148}.

\item[(c)] Unlike in the case of 
$\mathcal{E}_p^{\N_0}$, we do not have a fully explicit asymptotic
expansion of the Hardy weights. In particular, it is not clear
whether the lower-order terms are nonnegative.

\end{enumerate}
\end{remarks}

\begin{remark}
In the examples considered in
Subsections~\ref{subsec: Dirichlet Laplacian}
and~\ref{subsec: fractional Laplacian}, the Hardy inequalities hold on
the full natural weighted spaces determined by the corresponding Hardy
weights. Thus, within the respective weighted $\ell^p$ scales, their
domains are maximal. As shown below, in the more general setting, the maximality can fail.  
\end{remark}

Before closing the introduction, let us briefly outline the structure of the proofs. The argument separates into an abstract part and a model-specific part. First, an elementary algebraic identity for the Bregman divergence yields an exact ground-state representation for Sobolev--Bregman forms and, by dropping its nonnegative remainder, the corresponding Hardy inequality. The resulting Hardy weight is expressed through a positive superharmonic function and its powers.

The model-specific part consists in choosing suitable superharmonic functions and proving the criticality of the resulting weights. For the Dirichlet Laplacian, we use power functions on the half-line, whereas for the fractional Laplacian, we use the kernels $k_{-\alpha}$. Criticality is established by constructing appropriate null sequences for the remainder terms in the ground-state representations. In the fractional case, this requires additional estimates reflecting the nonlocal character of the kernel.

The paper is organized as follows. In Section~2, we introduce the general
framework, establish ground-state representations and Hardy inequalities
for Sobolev--Bregman forms on arbitrary discrete spaces, and explain
their connection with the contractivity of perturbed semigroups.
Section~3 applies these results to the Dirichlet Laplacian on the
half-line. In Section~4, we study the fractional Laplacian on $\Z$,
construct families of Hardy weights, and prove their criticality.

\section{General Setup and Ground-State Representation}\label{sec: abstract framework}

The purpose of this section is to derive Hardy inequalities for
$\mathcal{E}_p^{X,k}$ by means of \emph{ground-state representations}.

\subsection{Kernels and operators}\label{ss:ko}

Let $X$ be a countable set. By a kernel on
$X$ we mean a function
$$
m:X\times X\to\R.
$$
A kernel $m$ is called \emph{nonnegative} if
$$
m(x,y)\geq0,
\qquad x,y\in X,
$$
and \emph{symmetric} if
$$
m(x,y)=m(y,x),
\qquad x,y\in X.
$$

For a nonnegative kernel \(m\), we write
\begin{equation}\label{def: dk}
d_m(x)
:=
\sum_{y\in X}m(x,y),
\qquad x\in X.
\end{equation}
A kernel $m$ is called (row) \emph{summable} if
$$
\sum_{y\in X}|m(x,y)|<\infty,
\qquad x\in X,
$$
that is, $d_{|m|}<\infty$.
It is called \emph{uniformly summable} if
$$
\|m\|_{\mathrm{sum}}
:=
\sup_{x\in X}\sum_{y\in X}|m(x,y)|
<\infty,
$$
that is, $\sup d_{|m|}<\infty$.

Let $\delta$ denote the \emph{identity} kernel,
$$
\delta(x,y)
:=
\begin{cases}
1, & x=y,\\
0, & x\neq y.
\end{cases}
$$
We identify every bounded function $v:X\to\R$ with the \emph{diagonal}
kernel, denoted
$$
(v\delta)(x,y)
:=
v(x)\delta(x,y),
\qquad x,y\in X.
$$
Then
$$
\|v\delta\|_{\mathrm{sum}}
=
\|v\|_\infty.
$$

Every uniformly summable kernel $m$ defines a bounded operator
$T_m$ on $\ell^\infty(X)$ by
$$
T_mu(x)
:=
\sum_{y\in X}m(x,y)u(y),
\qquad x\in X.
$$
If, in addition, $m$ is symmetric, then $T_m$ is bounded on
$\ell^q(X)$ for every $q\in[1,\infty]$, 
\begin{equation}\label{kernel operator estimate}
\|T_m\|_{\ell^q\to\ell^q}
\leq
\|m\|_{\mathrm{sum}}.
\end{equation}

%For every function $v:X\to\R$, let $D_v$ denote the %diagonal kernel
%defined by
%$$
%D_v(x,y):=v(x)\delta(x,y),
%\qquad x,y\in X.
%$$
%The kernel $D_v$ is symmetric. It is uniformly %summable if and only if
%$v\in\ell^\infty(X)$:
%$$
%\|D_v\|_{\mathrm{sum}}=\|v\|_\infty.
%$$

For $v\in\ell^\infty(X)$, let $M_v$ denote the multiplication operator
defined by
$$
M_vu(x):=v(x)u(x),
\qquad x\in X.
$$
Then $M_v=T_{v\delta}$ is bounded on $\ell^q(X)$ for every $q\in[1,\infty]$, and
$$
\|M_v\|_{\ell^q\to\ell^q}=\|v\|_\infty.
$$
Throughout the paper, unless explicitly stated otherwise, we fix a nonnegative
symmetric summable kernel \(k:X\times X\to[0,\infty)\). 
Thus \(d_k(x)<\infty\) for every \(x\in X\). Uniform summability will be imposed only in those results
where boundedness on \(\ell^q\)-spaces is needed.

The Laplacian associated with $k$ is defined, whenever the sum
converges absolutely, by
\begin{equation}\label{def: abstract Laplacian}
L_ku(x)
:=
\sum_{y\in X}
\bigl(u(x)-u(y)\bigr)k(x,y),
\qquad x\in X.
\end{equation}
Equivalently,
$$
L_ku(x)
=
d_k(x)u(x)
-
\sum_{y\in X}k(x,y)u(y).
$$

If $k$ is uniformly summable, then $d_k\in\ell^\infty(X)$ and
$$
\|d_k\|_\infty
=
\|k\|_{\mathrm{sum}}.
$$
In this case, $L_k$ is associated with the symmetric uniformly
summable kernel $d_k\delta-k$:
\begin{equation}\label{Lk as kernel operator}
L_k
=
T_{d_k\delta-k}
=
M_{d_k}-T_k.
\end{equation}
Consequently, $L_k$ is bounded on $\ell^q(X)$ for every
$q\in[1,\infty]$, and
\begin{equation}\label{Lk operator norm estimate}
\|L_k\|_{\ell^q\to\ell^q}
\leq
2\|k\|_{\mathrm{sum}}.
\end{equation}

\subsection{Sobolev--Bregman form}\label{ss:SB}
Let $p\in(1,\infty)$. For $u:X\to\R$, the Sobolev--Bregman form
$\mathcal{E}_p^{X,k}(u)$ is defined by \eqref{def: SB}, with values
in $[0,\infty]$. The first step in understanding the Sobolev--Bregman form is to
observe the following pointwise equivalence valid for all $a,b\in\R$,
\begin{equation}\label{equivalence of 2 and p form}
\frac{4(p-1)}{p^2}
\bigl(b^{\langle p/2\rangle}-a^{\langle p/2\rangle}\bigr)^2
\leq
(b-a)\bigl(b^{\langle p-1\rangle}-a^{\langle p-1\rangle}\bigr)
\leq
2\bigl(b^{\langle p/2\rangle}-a^{\langle p/2\rangle}\bigr)^2,
\end{equation}
see \cite[Lemma 2.1]{MR4885983}, \cite{MR1407327}.
Consequently,
$$
\frac{4(p-1)}{p^2}
\mathcal{E}_2^{X,k}\bigl(u^{\langle p/2\rangle}\bigr)
\leq
\mathcal{E}_p^{X,k}(u)
\leq
2\mathcal{E}_2^{X,k}\bigl(u^{\langle p/2\rangle}\bigr).
$$
This comparison provides a way of constructing Hardy weights for the
nonlinear $p$-form from Hardy weights for the quadratic $2$-form.
However, for $p\neq2$, this procedure does not in general produce
critical Hardy weights for the $p$-form; see, for instance, the discussion of
\cite[(9)]{MR4372148}. A direct analysis of the nonlinear form is therefore
needed.

To this end, we introduce the \emph{Bregman divergence}
\begin{equation}
F_p(a,b)
:=
|b|^p-|a|^p-pa^{\langle p-1\rangle}(b-a),
\qquad a,b\in\R.
\end{equation}
By the convexity of the function $\R\ni x\mapsto|x|^p$, we have
$F_p\geq0$. In particular,
$$
F_2(a,b)=(b-a)^2.
$$
Symmetrizing $F_p$ in $a$ and $b$ gives
\begin{equation}\label{symmetrization of Fp}
\frac{1}{2}\bigl(F_p(a,b)+F_p(b,a)\bigr)
=
\frac{p}{2}
(b-a)\bigl(b^{\langle p-1\rangle}-a^{\langle p-1\rangle}\bigr).
\end{equation}
Consequently, the Sobolev--Bregman form admits the representation
\begin{equation}\label{Bregman representation}
\mathcal{E}_p^{X,k}(u)
=
\frac{1}{p}
\sum_{x\in X}\sum_{y\in X}
F_p\bigl(u(x),u(y)\bigr)k(x,y).
\end{equation}
% This identity has played a crucial role in the analysis of
% Sobolev--Bregman forms, particularly in the derivation of Hardy
% inequalities; see, for instance, \cite{MR4372148}.  
We record two auxiliary
properties of $F_p$ that will be used later.

\begin{lemma}\label{lem: from Fp to F2}
Let $p\in (1,\infty)$. Then, for all $a,b\in\R$,
$$
F_p\left(a^{\langle 2/p\rangle},
         b^{\langle 2/p\rangle}\right)
\leq
2p(a-b)^2.
$$
\end{lemma}
\begin{proof}
The result follows by \eqref{symmetrization of Fp} and \eqref{equivalence of 2 and p form}.
\end{proof}

The following elementary identity was first observed in
\cite{BogdanDydaSzczukiewicz2026}.
\begin{lemma}\label{lem: Bregman identity}
Let $a_1,a_2\in\R$ and $b_1,b_2\in(0,\infty)$. Then
\begin{align}
F_p(a_1,a_2)
&=
F_p\left(\frac{a_1}{b_1},\frac{a_2}{b_2}\right)
b_1^{p-1}b_2 \notag\\
&\quad
+(p-1)|a_1|^p\frac{b_1-b_2}{b_1}
+|a_2|^p
\frac{b_2^{p-1}-b_1^{p-1}}{b_2^{p-1}}.
\label{Bregman identity}
\end{align}
\end{lemma}
\begin{proof}
The identity follows by expanding the right-hand side according to the
definition of $F_p$ and collecting terms.
\end{proof}

For further estimates of $F_p$ and its applications in analysis, statistical
learning, and probability theory we refer to
\cite{MR4589708,MR4851904,bogdan2024bregmanvariationsemimartingales}.

\subsection{Ground state representations and Hardy inequalities}
Broadly speaking, a ``ground state representation'' is an identity that decomposes an energy form into a weighted term and a nonnegative remainder.
The idea goes back to classical linear
Schr\"odinger operators in the continuum and has found applications in,
among other areas, spectral theory, mathematical physics, and
Agmon--Allegretto--Piepenbrink theory
\cite{davies1989heat,pinchover2007topics,Birman1966Spectrum}.
Since then, ground state representations in the continuum have been
obtained in various settings; see, for instance, \cite{MR2469027}.
In the discrete setting, such representations already appear for
Jacobi matrices in the linear case \cite{MR2415477}, and were subsequently
obtained for linear Schrödinger operators on graphs \cite{MR3774437} and
quasilinear Schrödinger operators \cite{MR4768506}.

The main result of this section is a ground-state representation for the
discrete Sobolev--Bregman form \eqref{def: SB}. To state it, we introduce
some notation. 

Throughout this subsection, $\Omega\subseteq X$ denotes the region on which
the ground-state representation and the corresponding Hardy inequality are
considered. Dirichlet boundary conditions outside $\Omega$ will be encoded by
identifying functions on $\Omega$ with their extensions by zero to
$X\setminus\Omega$. We denote by $C_c(\Omega)$ the space of finitely
supported functions on $\Omega$, understood according to this convention.

Recall that the Laplacian $L_k$ associated with $k$ is
defined in \eqref{def: abstract Laplacian}. It is well defined and finite pointwise
on the space
\begin{equation}\label{def: domain of the formal Laplacian}
    \mathcal{F}_k(X) :=   \left\{ u: X\to\R: \sum_{y\in X}k(x,y)|u(y)|<\infty \text{ for every }x\in X \right\}.
\end{equation}
The summability of $k$ implies that
$\ell^\infty(X)\subseteq\mathcal{F}_k(X)$. Since $\ell^p(X)\subseteq\ell^\infty(X)$,

$$
\ell^p(X)\subseteq\mathcal{F}_k(X).
$$

A function $u\in\mathcal{F}_k(X)$ is called
\emph{harmonic} on $\Omega$ if
$$
L_ku(x)=0,
\qquad x\in\Omega,
$$
\emph{superharmonic} on $\Omega$ if
$$
L_ku(x)\geq0,
\qquad x\in\Omega,
$$
%and
%A function $u\in\mathcal{F}(X)$ is called
%\emph{harmonic} on $\Omega$ if
%$$
%L_ku(x)=0,
%\qquad x\in\Omega, $$ 
and \emph{strictly superharmonic} on $\Omega$ if $$L_ku(x)>0,\qquad x\in\Omega.$$

\begin{definition}
Let $\Omega\subseteq X$. A \emph{Hardy weight} for
$\mathcal{E}_p^{X,k}$ on $\Omega$ is a nonnegative function
$w:\Omega\to[0,\infty)$ such that
$$
\mathcal{E}_p^{X,k}(u)
\geq
\sum_{x\in\Omega}w(x)|u(x)|^p,
\qquad u\in C_c(\Omega).
$$
A Hardy weight $w$ is called \emph{critical} if every Hardy weight
$\widetilde w$ on $\Omega$ satisfying
$$
\widetilde w(x)\geq w(x),
\qquad x\in\Omega,
$$
coincides with $w$.
\end{definition}

Let $\Omega \subseteq X$, and $\mu : \Omega \rightarrow [0, \infty)$ be a non-negative weight, we define
$$
\ell^p(\Omega, \mu) := \Big\{u : X \rightarrow \R : \operatorname{supp} u  \subseteq \Omega\, \operatorname{and} \sum_{x \in \Omega} |u(x)|^p \mu(x) < \infty  \Big\}.
$$
Thus, elements of $\ell^p(\Omega,\mu)$ may equivalently be regarded as
functions on $\Omega$, extended by zero to $X\setminus\Omega$ according to
the convention above.

For a function $h:X\to[0,\infty)$ such that $h>0$ on $\Omega$,
$h=0$ on $X\setminus\Omega$, and
$h,h^{p-1}\in\mathcal F_k(X)$, define the weight
\begin{equation}\label{abstract measure}
\begin{aligned}
\mu_h(x)
&:=
\frac{1}{p'}
\frac{1}{h(x)}
\sum_{y\in X}|h(x)-h(y)|k(x,y)\\
&\quad
+
\frac{1}{p}
\frac{1}{h(x)^{p-1}}
\sum_{y\in X}
|h(x)^{p-1}-h(y)^{p-1}|k(x,y),
\qquad x\in\Omega.
\end{aligned}
\end{equation}
By the definition, the weight is finite and nonnegative. It specifies the class of admissible functions in the
ground-state representation below.
\begin{theorem}[Ground-state representation]
\label{thm: abstract ground state representation}
Let $p>1$ and $\Omega\subseteq X$. Let $h:X\to[0,\infty)$ satisfy
$h>0$ on $\Omega$ and $h=0$ on $X\setminus\Omega$. Suppose that
$h,h^{p-1}\in\mathcal F_k(X)$ and that both functions are superharmonic
on $\Omega$. Then, for every $u\in\ell^p(\Omega,\mu_h)$,
\begin{align}
\mathcal E_p^{X,k}(u)
&=
\sum_{x\in\Omega}w_h(x)|u(x)|^p \notag\\
&\quad
+
\frac{1}{p}
\sum_{x,y\in\Omega}
F_p\left(
\frac{u(x)}{h(x)},
\frac{u(y)}{h(y)}
\right)
h(x)^{p-1}h(y)k(x,y),
\label{abstract ground state representation}
\end{align}
where
\begin{equation}\label{abstract Hardy weight}
w_h(x)
:=
\frac{1}{p'}
\frac{L_kh(x)}{h(x)}
+
\frac{1}{p}
\frac{L_kh^{p-1}(x)}{h(x)^{p-1}},
\qquad x\in\Omega.
\end{equation}
Here and below, identities and inequalities between nonnegative quantities are understood in
\([0,\infty]\).
Moreover, $w_h$ is a Hardy weight for
$\mathcal E_p^{X,k}$ on $\Omega$ and
\begin{equation}\label{abstract Hardy inequality}
\mathcal E_p^{X,k}(u)
\geq
\sum_{x\in\Omega}w_h(x)|u(x)|^p,
\qquad u\in\ell^p(\Omega,\mu_h).
\end{equation}
\end{theorem}

\begin{proof}
Since $u=0$ on $X\setminus\Omega$, by \eqref{Bregman representation} we decompose the form as
\begin{equation}\label{decomposition into omega and its complement}
\mathcal E_p^{X,k}(u)=I_1(u)+I_2(u),
\end{equation}
where
$$
I_1(u)
:=
\frac{1}{p}
\sum_{x,y\in\Omega}
F_p\bigl(u(x),u(y)\bigr)k(x,y).
$$
and
$$
I_2(u)
:=
\sum_{x\in\Omega}
\sum_{y\in X\setminus\Omega}
|u(x)|^p k(x,y).
$$
We now apply \eqref{Bregman identity} with
$$
a_1=u(x),\qquad a_2=u(y),\qquad
b_1=h(x),\qquad b_2=h(y).
$$
This gives
\begin{equation}\label{decomposition of the interior term}
I_1(u)
=
\frac{1}{p}
\sum_{x,y\in\Omega}
\bigl(f_1(x,y)+f_2(x,y)+f_3(x,y)\bigr),
\end{equation}
where
$$
f_1(x,y)
:=
(p-1)|u(x)|^p
\frac{h(x)-h(y)}{h(x)}
k(x,y),
$$
$$
f_2(x,y)
:=
|u(y)|^p
\frac{h(y)^{p-1}-h(x)^{p-1}}
     {h(y)^{p-1}}
k(x,y),
$$
and
$$
f_3(x,y)
:=
F_p\left(
\frac{u(x)}{h(x)},
\frac{u(y)}{h(y)}
\right)
h(x)^{p-1}h(y)k(x,y).
$$

The assumption $u\in\ell^p(\Omega,\mu_h)$ implies that
$f_1,f_2\in\ell^1(\Omega\times\Omega)$. Hence, by Fubini's
theorem and the symmetry of $k$,
\begin{align}
I_1(u)
&=
\sum_{x\in\Omega}|u(x)|^p
\biggl[
\frac{1}{p'}
\frac{1}{h(x)}
\sum_{y\in\Omega}
\bigl(h(x)-h(y)\bigr)k(x,y)
\notag\\
&\hspace{42mm}
+
\frac{1}{p}
\frac{1}{h(x)^{p-1}}
\sum_{y\in\Omega}
\bigl(h(x)^{p-1}-h(y)^{p-1}\bigr)k(x,y)
\biggr]
\notag\\
&\quad
+
\frac{1}{p}
\sum_{x,y\in\Omega}
F_p\left(
\frac{u(x)}{h(x)},
\frac{u(y)}{h(y)}
\right)
h(x)^{p-1}h(y)k(x,y).
\label{expressing omega term}
\end{align}
The last sum is well defined in $[0,\infty]$ since
$F_p\geq0$. Since $h=h^{p-1}=0$ on $X\setminus\Omega$, 
\begin{align}
I_2(u)
&=
\sum_{x\in\Omega}|u(x)|^p
\sum_{y\in X\setminus\Omega}k(x,y)
\notag\\
&=
\frac{1}{p'}
\sum_{x\in\Omega}
\frac{|u(x)|^p}{h(x)}
\sum_{y\in X\setminus\Omega}
\bigl(h(x)-h(y)\bigr)k(x,y)
\notag\\
&\quad
+
\frac{1}{p}
\sum_{x\in\Omega}
\frac{|u(x)|^p}{h(x)^{p-1}}
\sum_{y\in X\setminus\Omega}
\bigl(h(x)^{p-1}-h(y)^{p-1}\bigr)k(x,y).
\label{expressing boundary term}
\end{align}

Combining \eqref{expressing omega term} and
\eqref{expressing boundary term} with
\eqref{decomposition into omega and its complement}, and using
the definition of $L_k$, yields
\begin{align*}
\mathcal E_p^{X,k}(u)
&=
\sum_{x\in\Omega}
\left(
\frac{1}{p'}\frac{L_kh(x)}{h(x)}
+
\frac{1}{p}
\frac{L_kh^{p-1}(x)}{h(x)^{p-1}}
\right)
|u(x)|^p\\
&\quad
+
\frac{1}{p}
\sum_{x,y\in\Omega}
F_p\left(
\frac{u(x)}{h(x)},
\frac{u(y)}{h(y)}
\right)
h(x)^{p-1}h(y)k(x,y),
\end{align*}
which is \eqref{abstract ground state representation}.

Since $F_p\geq0$, the representation implies
\eqref{abstract Hardy inequality}. Moreover, the superharmonicity
of $h$ and $h^{p-1}$ shows that $w_h\geq0$ on $\Omega$.
Finally, $\mu_h$ is finite at every point of $\Omega$, and hence
$C_c(\Omega)\subseteq\ell^p(\Omega,\mu_h)$. Therefore, $w_h$ is
a Hardy weight for $\mathcal E_p^{X,k}$ on $\Omega$.
\end{proof}

\begin{remarks}\label{rem: reference measure}
\begin{enumerate}

\item[(a)]
For $p=2$, we recover the classical ground-state representation and
the associated Hardy inequality for the $\ell^2$-form
$\mathcal E_2^{X,k}$. In particular, the corresponding Hardy weight
reduces to the familiar expression \cite{MR3774437}:
$$
w_h=\frac{L_kh}{h}.
$$

\item[(b)]
The proof of the ground-state representation is quite flexible, as it
relies on the algebraic identity \eqref{Bregman identity}. It can
therefore be adapted to other settings in which suitable summability properties are available.

\item[(c)]
The weight $\mu_h$ is closely related to the Hardy weight $w_h$.
Indeed, the definition of $L_k$ and the triangle inequality give
$$
w_h(x)\leq\mu_h(x),
\qquad x\in\Omega.
$$
It is therefore natural to ask whether the assumption
$u\in\ell^p(\Omega,\mu_h)$ in the ground-state representation can be
weakened to $u\in\ell^p(\Omega,w_h)$. In general, this is not possible.

To see this, let $p=2$, and let $k$ be a nearest-neighbor kernel which induces on $X$ a
transient graph structure. Set $h=G_o$, where $G_o$ is a positive Green
function of $L_k$ with pole at $o\in X$. Since
$$
L_kG_o=\delta_o,
$$
the corresponding Hardy weight is
$$
w_h(x)
=
\frac{\delta_o(x)}{G_o(x)}
=
\frac{\delta_o(x)}{G_o(o)}.
$$
Thus, $w_h$ is supported at $o$, and the constant function
$u\equiv1$ belongs to $\ell^2(X,w_h)$. If membership in this space
were sufficient for the Hardy inequality, we would obtain
$$
0
=
\mathcal E_2^{X,k}(1)
\geq
\sum_{x\in X}w_h(x)
=
\frac{1}{G_o(o)}
>0,
$$
which is impossible. Hence, membership in
$\ell^p(\Omega,w_h)$ alone is not sufficient for either the
ground-state representation or the Hardy inequality
\eqref{abstract Hardy inequality}.

\item[(d)]
Since $w_h\leq\mu_h$, one always has
$$
\ell^p(\Omega,\mu_h)
\subseteq
\ell^p(\Omega,w_h).
$$
The preceding example shows that this inclusion may be strict. Thus,
the admissible domain furnished by the abstract ground-state
representation may be strictly smaller than the full weighted space
determined by the corresponding Hardy weight. The maximal-domain
property established in Subsections~\ref{subsec: Dirichlet Laplacian}
and \ref{subsec: fractional Laplacian} is therefore a special feature
of those model cases and does not follow from the abstract theory
alone.

\item[(e)]
In many settings of interest, the form $\mathcal E_p^{X,k}$ admits no nonzero Hardy
weight---we then say it is not \emph{subcritical}. Subcriticality can often be restored by imposing additional
vanishing conditions on the admissible functions, for example by
requiring them to be supported in a proper subset
$\Omega\subsetneq X$. The formulation above accommodates the
resulting Dirichlet boundary conditions. In
Section~\ref{sec: Dirichlet Laplacian}, we see how this phenomenon arises
for the Dirichlet Laplacian on the discrete half-line.
\end{enumerate}
\end{remarks}

\begin{proposition}\label{prop: abstract ground state representation for lp}
Let $h$ be as in Theorem~\ref{thm: abstract ground state representation}. If  
$$
K
:=
\sup_{x\in\Omega}\sum_{y\in X}k(x,y)
<\infty,
$$
then
$$
\mu_h(x)\leq 2K,
\qquad x\in\Omega.
$$
In particular,
$
\ell^p(\Omega)
\subseteq
\ell^p(\Omega,\mu_h)
$ and the ground-state representation
\eqref{abstract ground state representation}
holds for all $u\in\ell^p(\Omega)$.
\end{proposition}

\begin{proof}
Let $\varphi$ denote either $h$ or $h^{p-1}$. Since $\varphi$ is nonnegative and superharmonic on $\Omega$, for every $x\in\Omega$,
$$
\sum_{y\in X}\varphi(y)k(x,y)
\leq
\varphi(x)\sum_{y\in X}k(x,y)
\leq
K\varphi(x).
$$
It follows that
\begin{align*}
\frac{1}{\varphi(x)}
\sum_{y\in X}
|\varphi(x)-\varphi(y)|k(x,y)
&\leq
\sum_{y\in X}k(x,y)
+
\frac{1}{\varphi(x)}
\sum_{y\in X}\varphi(y)k(x,y)\leq 2K.
\end{align*}
Applying this estimate to both $h$ and $h^{p-1}$ and using
$\frac{1}{p}+\frac{1}{p'}=1$, we obtain
$\mu_h(x)\leq 2K$. The remaining assertions follow immediately.
\end{proof}

For $p=2$, we have $h^{p-1}=h$, so a single positive superharmonic
function $h$ suffices to produce the Hardy weight $w_h=L_kh/h\geq 0$.
For general $p$, the situation seems more involved because both $h$ and
$h^{p-1}$ are required to be superharmonic. However, the ground-state representation has the following equivalent
one-function formulation. Let
\[
    \nu_g:=\mu_{g^{q_p}}.
\]
Equivalently,
\begin{align}
    \nu_g(x)
    :={}&
    \frac{a_p}{a_p+1}
    \frac{1}{g(x)}
    \sum_{y\in X}
    |g(x)-g(y)|k(x,y)
    \nonumber\\
    &+
    \frac{1}{a_p+1}
    \frac{1}{g(x)^{a_p}}
    \sum_{y\in X}
    \left|g(x)^{a_p}-g(y)^{a_p}\right|k(x,y),
    \qquad x\in X.
    \label{abstract one-function reference measure}
\end{align}

\begin{proposition}\label{prop: abstract ground state representation: one-function formulation}
Let $p>1$ and $\Omega\subseteq X$. Let $g\in\mathcal F_k(X)$ be
nonnegative, positive on $\Omega$, zero on $X\setminus\Omega$, and
superharmonic on $\Omega$. Define
$$
h(x)
:=
\begin{cases}
g(x), & 1<p\leq2,\\[2mm]
g(x)^{1/(p-1)}, & p>2.
\end{cases}
$$
Then $h$ and $h^{p-1}$ belong to $\mathcal F_k(X)$ and are
superharmonic on $\Omega$. For all
$u\in\ell^p(\Omega,\mu_h)$,
\begin{align}
\mathcal E_p^{X,k}(u)
&=
\sum_{x\in\Omega}w_h(x)|u(x)|^p
\notag\\
&\quad+
\frac{1}{p}
\sum_{x,y\in\Omega}
F_p\left(
\frac{u(x)}{h(x)},
\frac{u(y)}{h(y)}
\right)
h(x)^{p-1}h(y)k(x,y),
\label{abstract ground state representation: one-function formulation}
\end{align}
and
$$
\mathcal E_p^{X,k}(u)
\geq
\sum_{x\in\Omega}w_h(x)|u(x)|^p,
$$
where $w_h$ and $\mu_h$ are given by
\eqref{abstract Hardy weight} and \eqref{abstract measure}, respectively.
\end{proposition}

\begin{proof}
We have
$$
\{h,h^{p-1}\}=\{g,g^{a_p}\}.
$$
It remains to justify the superharmonicity of \(g^{a_p}\).
Since \(0<a_p\leq1\), we have \(g^{a_p}\leq 1+g\), and hence
\(g^{a_p}\in\mathcal F_k(X)\). If
\(d_k(x)=0\), then \(L_kg^{a_p}(x)=0\). If \(d_k(x)>0\), then the
superharmonicity of \(g\) gives
\[
    \frac{1}{d_k(x)}\sum_{y\in X}g(y)k(x,y)
    \leq
    g(x).
\]
Since \(t\mapsto t^{a_p}\) is concave on \([0,\infty)\), Jensen's inequality
yields
\[
    \frac{1}{d_k(x)}\sum_{y\in X}g(y)^{a_p}k(x,y)
    \leq
    \left(
        \frac{1}{d_k(x)}\sum_{y\in X}g(y)k(x,y)
    \right)^{a_p}
    \leq
    g(x)^{a_p}.
\]
Thus \(L_kg^{a_p}(x)\geq0\).

Hence both $h$ and $h^{p-1}$ satisfy the assumptions of
Theorem~\ref{thm: abstract ground state representation}, and the
assertions follow.
\end{proof}
\begin{remark}\label{rem: ground-state in terms of constants a_p}
The Hardy weight, the weight $\mu_h$, and the remainder term in the
proposition admit convenient expressions in terms of
$$
a_p=(p-1)\wedge(p'-1).
$$
Indeed, let
$$
q_p:=1\wedge\frac{1}{p-1},
\qquad
r_p:=1\wedge(p-1).
$$
Then
$$
h=g^{q_p},
\qquad
h^{p-1}=g^{r_p},
$$
and
$$
\{q_p,r_p\}=\{1,a_p\},
\qquad
q_p+r_p=1+a_p.
$$
Moreover,
$$
\frac{1}{p}=\frac{q_p}{q_p+r_p},
\qquad
\frac{1}{p'}=\frac{r_p}{q_p+r_p}.
$$
Substituting these identities into the definition of $w_h$, we obtain
$$
w_h
=
\frac{1}{a_p+1}
\frac{L_k g^{a_p}}{g^{a_p}}
+
\frac{a_p}{a_p+1}
\frac{L_k g}{g}.
$$
Similarly,
\begin{align*}
\mu_h(x)
={}&
\frac{1}{a_p+1}
\frac{1}{g(x)^{a_p}}
\sum_{y\in X}
\left|g(x)^{a_p}-g(y)^{a_p}\right|k(x,y)
\\
&+
\frac{a_p}{a_p+1}
\frac{1}{g(x)}
\sum_{y\in X}
|g(x)-g(y)|k(x,y).
\end{align*}
The remainder term in the ground-state representation becomes
$$
\frac{1}{p}
\sum_{x,y\in\Omega}
F_p\left(
\frac{u(x)}{g(x)^{q_p}},
\frac{u(y)}{g(y)^{q_p}}
\right)
g(x)^{r_p}g(y)^{q_p}k(x,y).
$$
Finally, the exponents admit the expressions
$$
q_p=\sqrt{a_p(p'-1)},
\qquad
r_p=\sqrt{a_p(p-1)}.
$$
\end{remark}

\subsection{Semigroups and contractivity}\label{sec:c}
This subsection presents one of the main motivations for studying Hardy
inequalities for Sobolev--Bregman forms. We first introduce the composition
and exponential of uniformly summable kernels and use them to construct
semigroups generated by bounded perturbations of $-L_k$. We then show that
a Hardy inequality for the Sobolev--Bregman form is equivalent to the
contractivity of the corresponding perturbed semigroup on $\ell^p(X)$.
The argument follows the strategy of \cite{MR4372148}, but we include a
self-contained proof, which is particularly short and direct in the present
setting.

The composition of two uniformly summable kernels $m$ and $n$ is
defined by
\begin{equation}\label{e.dck}
mn(x,y)
:=
\sum_{z\in X}m(x,z)n(z,y),
\qquad x,y\in X.
\end{equation}
The composition is well defined, is again a uniformly summable
kernel, and satisfies
\begin{equation}\label{kernel product estimate}
\|mn\|_{\mathrm{sum}}
\leq
\|m\|_{\mathrm{sum}}\|n\|_{\mathrm{sum}}.
\end{equation}
Notice that the composition of two symmetric kernels need not be
symmetric.

The identity kernel $\delta$ defined in Subsection~\ref{ss:ko}
satisfies
$$
\delta m=m\delta=m.
$$
Moreover, if $v:X\to\R$ is bounded, then
$$
\bigl((v\delta)m\bigr)(x,y)
=
v(x)m(x,y)
$$
and
$$
\bigl(m(v\delta)\bigr)(x,y)
=
m(x,y)v(y).
$$
In particular,
\begin{equation}\label{bounded function kernel product estimate}
\|(v\delta)m\|_{\mathrm{sum}}
\leq
\|v\|_\infty\|m\|_{\mathrm{sum}},
\qquad
\|m(v\delta)\|_{\mathrm{sum}}
\leq
\|m\|_{\mathrm{sum}}\|v\|_\infty.
\end{equation}

For a uniformly summable kernel $m$, define
$$
m^0:=\delta,
\qquad
m^{j+1}:=m^jm,
\qquad j\in\N_0.
$$
The exponential of $m$ is defined by
\begin{equation}\label{def: kernel exponential}
e^{tm}
:=
\sum_{j=0}^{\infty}\frac{t^j}{j!}m^j,
\qquad t\in\R.
\end{equation}
The series converges with respect to
$\|\cdot\|_{\mathrm{sum}}$, and
\begin{equation}\label{kernel exponential sum estimate}
\|e^{tm}\|_{\mathrm{sum}}
=
\sup_{x\in X}
\sum_{y\in X}|e^{tm}(x,y)|
\leq
e^{|t|\|m\|_{\mathrm{sum}}}.
\end{equation}
If $m$ is symmetric, then $m^j$ is symmetric for every
$j\in\N_0$, and hence $e^{tm}$ is symmetric.

For uniformly summable kernels $m$ and $n$, we have
$$
T_{mn}=T_mT_n
$$
on $\ell^\infty(X)$. Consequently,
\begin{equation}\label{kernel and operator exponentials}
T_{e^{tm}}
=
e^{tT_m}.
\end{equation}
If $m$ is symmetric, then this identity holds on every
$\ell^q(X)$, $q\in[1,\infty]$, and
\begin{equation}\label{kernel exponential operator estimate}
\|T_{e^{tm}}\|_{\ell^q\to\ell^q}
\leq
\|e^{tm}\|_{\mathrm{sum}}
\leq
e^{|t|\|m\|_{\mathrm{sum}}}.
\end{equation}

Assume now that $k$ is uniformly summable and let
$w\in\ell^\infty(X)$. The \emph{perturbed operator}
$$
-L_k+M_w
$$
is associated with the symmetric uniformly summable kernel
\begin{equation}\label{def: perturbed kernel}
a:=k+(w-d_k)\delta.
\end{equation}
Indeed, by \eqref{Lk as kernel operator},
\begin{equation}\label{perturbed operator as kernel operator}
-L_k+M_w=T_a
\end{equation}
Furthermore,
\begin{equation}\label{perturbed kernel sum estimate}
\|a\|_{\mathrm{sum}}
\leq
2\|k\|_{\mathrm{sum}}
+
\|w\|_\infty.
\end{equation}

For $t\geq0$, define
\begin{equation}\label{def: perturbed semigroup}
P_t^{k,w}
=
e^{t(-L_k+M_w)}
=
T_{e^{ta}}.
\end{equation}
Then $(P_t^{k,w})_{t\geq0}$ is a semigroup of bounded operators on
$\ell^q(X)$ for every $q\in[1,\infty]$. Its kernel satisfies
\begin{equation}\label{perturbed semigroup kernel estimate}
\sup_{x\in X}
\sum_{y\in X}
\bigl|e^{ta}(x,y)\bigr|
\leq
\exp\left(
t\bigl(
2\|k\|_{\mathrm{sum}}
+
\|w\|_\infty
\bigr)
\right),
\end{equation}
and hence
\begin{equation}\label{perturbed semigroup operator estimate}
\|P_t^{k,w}\|_{\ell^q\to\ell^q}
\leq
\exp\left(
t\bigl(
2\|k\|_{\mathrm{sum}}
+
\|w\|_\infty
\bigr)
\right),
\qquad q\in[1,\infty].
\end{equation}

We now characterize Hardy inequalities for the Sobolev--Bregman form
in terms of the contractivity of the corresponding perturbed semigroups.

\begin{theorem}\label{thm: contractivity of perturbed semigroup}
Let \(p\in(1,\infty)\), let \(k\) be uniformly summable, and let
\(w\in\ell^\infty(X)\) be real-valued. Then the following statements are
equivalent:
\begin{enumerate}
\item For every \(u\in\ell^p(X)\),
\begin{equation}\label{inequality for contractivity}
    \mathcal E_p^{X,k}(u)
    \geq
    \sum_{x\in X} w(x)|u(x)|^p.
\end{equation}
\item The semigroup \((P_t^{k,w})_{t\geq0}\) is contractive on
\(\ell^p(X)\).
\end{enumerate}
\end{theorem}
\begin{proof}
For $f\in\ell^p(X)$, the symmetry of $k$ gives
\begin{align}
\left\langle L_kf,f^{\langle p-1\rangle}\right\rangle
&=
\frac{1}{2}
\sum_{x\in X}\sum_{y\in X}
\bigl(f(x)-f(y)\bigr)
\bigl(
f(x)^{\langle p-1\rangle}
-
f(y)^{\langle p-1\rangle}
\bigr)
k(x,y)
\nonumber\\
&=
\E_p^{X,k}(f).
\label{Green identity for SB form}
\end{align}
The double sum is absolutely convergent by the uniform summability of
$k$ and H\"older's inequality.
\ifdetails
Indeed, by H\"older's inequality and the symmetry and uniform
summability of $k$,
\begin{align*}
&\sum_{x,y\in X}
|f(x)-f(y)|
\left|
f(x)^{\langle p-1\rangle}
-f(y)^{\langle p-1\rangle}
\right|
k(x,y)\\
&\qquad\leq
4\|k\|_{\mathrm{sum}}\|f\|_p^p<\infty.
\end{align*}
Hence the double sum in \eqref{Green identity for SB form} is
absolutely convergent.
\fi
Assume \textup{(1)}.
Fix $u\in\ell^p(X)$ and set
$$
g(t):=P_t^{k,w}u.
$$
Since $A:=-L_k+M_w$ is bounded on $\ell^p(X)$, the mapping
$t\mapsto g(t)$ is continuously differentiable in $\ell^p(X)$ and
$$
g'(t)=Ag(t).
$$
Moreover, the mapping $f\mapsto\|f\|_p^p$ is continuously Fr\'echet
differentiable on $\ell^p(X)$, with derivative
$$
D\bigl(\|\cdot\|_p^p\bigr)(f)h
=
p\left\langle h,f^{\langle p-1\rangle}\right\rangle,
$$
see \eqref{e.dpf}.
Consequently, by \eqref{Green identity for SB form},
\begin{align*}
\frac{d}{dt}\|g(t)\|_p^p
&=
p\left\langle
g'(t),g(t)^{\langle p-1\rangle}
\right\rangle\\
&=
-p\left(
\E_p^{X,k}(g(t))
-
\sum_{x\in X}w(x)|g(t,x)|^p
\right)
\leq0.
\end{align*}
Thus $t\mapsto\|P_t^{k,w}u\|_p^p$ is nonincreasing, and hence
$(P_t^{k,w})_{t\geq0}$ is contractive on $\ell^p(X)$.

Conversely, assume \textup{(2)}, i.e., for all $u\in\ell^p(X)$ and $t\geq0$,
$$
\|P_t^{k,w}u\|_p^p
\leq
\|u\|_p^p.
$$
Therefore, 
\begin{align*}
0
&\geq
\left.\frac{d}{dt}\right|_{t=0}
\|P_t^{k,w}u\|_p^p=
p\left\langle
(-L_k+M_w)u,u^{\langle p-1\rangle}
\right\rangle\\
&=
-p\left(
\E_p^{X,k}(u)
-
\sum_{x\in X}w(x)|u(x)|^p
\right).
\end{align*}
This is equivalent to \eqref{inequality for contractivity},
which completes the proof.
\end{proof}

\section{Dirichlet Laplacian on the half-line}
\label{sec: Dirichlet Laplacian}

In this section, we specialize the abstract framework of
Section~\ref{sec: abstract framework} to the Dirichlet Laplacian on the
half-line and prove
Theorems~\ref{thm: critical Hardy weights for Dirichlet Laplacian}
and~\ref{thm: sharp inverse-square Hardy}. We take
\[
    X=\N_0,\qquad \Omega=\N.
\]
Throughout this section, functions on \(\N\) are identified with their zero
extensions to \(\N_0\); thus we can also write \(u(0)=0\).

For such functions we define the positive discrete Dirichlet Laplacian by
\[
    \Delta u(x)
    :=
    2u(x)-u(x-1)-u(x+1),
    \qquad x\in\N.
\]

The corresponding nearest-neighbour Sobolev--Bregman form is
\begin{align}\label{Dirichlet Laplacian p-form_2}
    \mathcal{E}_{p}^{\N_0}(u)
    &:=
    \frac12
    \sum_{\substack{x,y\in\N_0\\ |x-y|=1}}
    \bigl(u(x)-u(y)\bigr)
    \bigl(
        u(x)^{\langle p-1\rangle}
        -
        u(y)^{\langle p-1\rangle}
    \bigr) \nonumber\\
    &=
    \sum_{x=1}^{\infty}
    \bigl(u(x)-u(x-1)\bigr)
    \bigl(
        u(x)^{\langle p-1\rangle}
        -
        u(x-1)^{\langle p-1\rangle}
    \bigr).
\end{align}
Thus \(\mathcal E_p^{\N_0}\) is precisely the Sobolev--Bregman form
associated with the nearest-neighbour kernel on \(X=\N_0\), with the
Dirichlet condition encoded by the zero extension outside
\(\Omega=\N\).

For later use in connection with perturbed semigroups, and also to give
an alternative perspective on the role of \(X\) and \(\Omega\) in
Theorem~\ref{thm: abstract ground state representation}, it is useful to view
\(\Delta\) directly as an operator on \(\ell^p(\N)\). Let
\[
    \kappa(x,y):=\mathbf 1_{\{|x-y|=1\}},
    \qquad x,y\in\N,
\]
and
\[
    V_{\partial}(x):=\mathbf 1_{\{1\}}(x).
\]
Then
\[
    \Delta=L_\kappa+M_{V_{\partial}}.
\]
Indeed, \(L_\kappa u(x)=\Delta u(x)\) for \(x\geq2\), whereas
\[
    L_\kappa u(1)=u(1)-u(2), \qquad \Delta u(1)=2u(1)-u(2).
\]
Correspondingly,
\begin{equation}\label{e.rbfN0}
        \mathcal E_p^{\N_0}(u)
    =
    \mathcal E_p^{\N,\kappa}(u)+|u(1)|^p.
\end{equation}
Thus, when the problem is formulated directly on \(\N\), the Dirichlet
boundary condition is represented by the additional bounded potential
\(V_{\partial}\), concentrated at \(1\). In the formulation
\(X=\N_0\), \(\Omega=\N\), the same boundary contribution is encoded by
the zero extension to \(X\setminus\Omega=\{0\}\) and the nearest-neighbour
kernel on \(X=\N_0\).

The proof of Theorem~\ref{thm: critical Hardy weights for Dirichlet Laplacian}
has two steps. First, we use the superharmonic functions
\[
    g(x)=x^\alpha,\qquad \alpha\in(0,1],
\]
and the ground-state representation
\eqref{abstract ground state representation: one-function formulation}
to obtain Hardy inequalities with the weights
\[
    w_{p,\alpha}(x)
    =
    \frac{1}{a_p+1}
    \frac{\Delta x^{\alpha a_p}}{x^{\alpha a_p}}
    +
    \frac{a_p}{a_p+1}
    \frac{\Delta x^\alpha}{x^\alpha},
    \qquad x\in\N,
\]
on the space \(\ell^p(\N,x^{-1})\). A regularization argument then extends
these inequalities to the larger space
\[
    \ell^p(\N,x^{-2})\supsetneq \ell^p(\N,x^{-1}).
\]
Second, for the critical range of \(\alpha\), we prove that the resulting Hardy
weights cannot be increased. This is done by approximating the corresponding
ground state with finitely supported functions and using the ground-state
representation again.
\begin{lemma}[Regularity]
\label{lem: truncation Dirichlet half-line}
Let \(p>1\), \(\alpha\in(0,1]\), and 
\(u\in\ell^p(\N,x^{-2})\). Then there exists a sequence
\((u_j)\) of finitely supported functions \(u_j:\N_0\to\R\), with
\(u_j(0)=0\), such that
\[
    \mathcal E_p^{\N_0}(u_j)
    \longrightarrow
    \mathcal E_p^{\N_0}(u),
\]
and
\[
    \sum_{x\in\N}|u_j(x)|^p w_{p,\alpha}(x)
    \longrightarrow
    \sum_{x\in\N}|u(x)|^p w_{p,\alpha}(x),
\]
as \(j\to\infty\). The limits are understood in \([0,\infty]\).
\end{lemma}

\begin{proof}
Since $u\in \ell^p(\N,x^{-2})$,
\[
\liminf_{x\to\infty}\frac{|u(x)|^p}{x}=0.
\]
Indeed, otherwise there would exist $\varepsilon>0$ and $N_0\in\N$ such that
\[
\frac{|u(x)|^p}{x}\geq \varepsilon,
\qquad x \geq N_0.
\]
It would then follow that
\[
\sum_{x\in\N}\frac{|u(x)|^p}{x^2}
\geq
\varepsilon\sum_{x=N_0}^\infty\frac1x
=\infty,
\]
contradicting $u\in \ell^p(\N,x^{-2})$. Therefore, there exists a
strictly increasing sequence $(N_j)_{j\in\N}$ such that
\[
\lim_{j\to\infty}\frac{|u(N_j)|^p}{N_j}=0.
\]
For every $j\in\N$, define 
\[
u_j(x):=
\begin{cases}
u(x), & 0\leq x < N_j,\\[1mm]
u(N_j)\left(2-\dfrac{x}{N_j}\right),
    & N_j\leq x < 2N_j,\\[2mm]
0, & 2 N_j \leq x.
\end{cases}
\]
Then \(u_j\) are finitely supported, \(u_j(0)=0\), and
\(u_j(x)\to u(x)\) for all $x\in\N_0$.
For $N_j<x\leq 2N_j$, we have
\[
u_j(x)-u_j(x-1)=-\frac{u(N_j)}{N_j}.
\]
Hence
\begin{align*}
&\sum_{x=N_j+1}^{2N_j}
\bigl(u_j(x)-u_j(x-1)\bigr)
\bigl(u_j(x)^{\langle p-1\rangle} -u_j(x-1)^{\langle p-1\rangle} \bigr)\\
&= \frac{|u(N_j)|^p}{N_j}
\sum_{x=N_j+1}^{2N_j}
\left(
\left( 2 - \frac{x-1}{N_j}\right)^{p-1}  -
\left( 2 - \frac{x}{N_j}\right)^{p-1}  \right)\\
&= \frac{|u(N_j)|^p}{N_j}.
\end{align*}
Therefore, by \eqref{Dirichlet Laplacian p-form_2},
\begin{align*}
\mathcal{E}_p^{\N_0}(u_j)
&=
\sum_{x=1}^{N_j}
\bigl(u(x)-u(x-1)\bigr)
\bigl(
u(x)^{\langle p-1\rangle}
-u(x-1)^{\langle p-1\rangle}
\bigr)
+
\frac{|u(N_j)|^p}{N_j}.
\end{align*}
Consequently, 
\begin{equation}\label{limit of the form}
    \lim_{j \rightarrow \infty} \mathcal{E}_p^{\N_0}(u_j) = \mathcal{E}_p^{\N_0}(u).
\end{equation}
For the Hardy-weight term, we split the sum according to the definition of
$u_j$ as
\begin{align*}
\sum_{x\in\N} |u_j(x)|^p w_{p,\alpha}(x)
&=
\sum_{x=1}^{N_j} |u(x)|^p w_{p,\alpha}(x)
+
\sum_{x=N_j+1}^{2N_j} |u_j(x)|^p w_{p,\alpha}(x).
\end{align*}
On the linear part of the truncation,
\[
|u_j(x)|
=
|u(N_j)|\left(2-\frac{x}{N_j}\right)
\leq |u(N_j)|,
\qquad N_j< x \leq 2N_j.
\]
Moreover, by \eqref{Taylor expansion of Dirichlet Hardy weight}, we have
\(w_{p,\alpha}(x)\lesssim x^{-2}\). 
Here and below, \(A(x)\lesssim B(x)\) means that there is a constant
\(C>0\), independent of \(x\), such that \(A(x)\leq C B(x)\). We write
\[A(x)\asymp B(x)\] if both \(A(x)\lesssim B(x)\) and
\(B(x)\lesssim A(x)\).
Therefore, we get
\begin{align*}
\sum_{x=N_j+1}^{2N_j}|u_j(x)|^p w_{p,\alpha}(x)
&\lesssim
|u(N_j)|^p
\sum_{x=N_j+1}^{2N_j}x^{-2}  \\
&\lesssim
\frac{|u(N_j)|^p}{N_j}.
\end{align*}
By the choice of the sequence $(N_j)$, the right-hand side converges to
zero as $j\to\infty$. Consequently,
\begin{equation}\label{limit of the Hardy term}
    \lim_{j \to \infty} \sum_{x\in\N} |u_j(x)|^p w_{p,\alpha}(x) = \sum_{x\in\N} |u(x)|^p w_{p,\alpha}(x).
\end{equation}
\end{proof}

\begin{proof}[Proof of Theorem \ref{thm: critical Hardy weights for Dirichlet Laplacian}]
Let \(\alpha\in(0,1]\) and define
\[
    g(x):=x^\alpha,\qquad x\in\N_0.
\]
Then \(g(0)=0\) and \(g(x)>0\) for \(x\in\N\). Since
\(t\mapsto t^\alpha\) is concave on \([0,\infty)\),
\[
    \Delta g(x)
    =
    2g(x)-g(x-1)-g(x+1)
    \geq0,
    \qquad x\in\N.
\]
Thus \(g\) is superharmonic on \(\N\). Moreover, if
\((p,\alpha)\neq(2,1)\), then at least one of the functions
\(x\mapsto x^\alpha\) and \(x\mapsto x^{\alpha a_p}\) is strictly concave on
\([0,\infty)\), and hence strictly superharmonic on \(\N\). It follows from
the definition of \(w_{p,\alpha}\) that
\[
    w_{p,\alpha}(x)>0,\qquad x\in\N.
\]

Proposition \ref{prop: abstract ground state representation: one-function formulation}
therefore applies with $g(x)=x^\alpha$ to every
$u\in\ell^p(\N,\mu_{p, \alpha})$ satisfying $u(0)=0$, where 
\begin{align*}
\mu_{p,\alpha}(x)
:=
&\frac{1}{a_p+1}\frac{1}{x^{\alpha a_p}}
\sum_{\substack{y\in\N_0\\ |x-y|=1}}
\left|x^{\alpha a_p}-y^{\alpha a_p}\right|
\\
&+
\frac{a_p}{a_p+1}\frac{1}{x^\alpha}
\sum_{\substack{y\in\N_0\\ |x-y|=1}}
\left|x^\alpha-y^\alpha\right|,
\qquad x\in\N.
\end{align*}
We claim that 
\[
\mu_{p,\alpha}(x)\asymp x^{-1}.
\]
Indeed, for any $\gamma\in(0,1]$ we have
\begin{align*}
\frac{1}{x^\gamma}
\sum_{\substack{y\in\N_0\\ |x-y|=1}}
|x^\gamma-y^\gamma|
&=
\frac{x^\gamma-(x-1)^\gamma+(x+1)^\gamma-x^\gamma}{x^\gamma}\\
&=
\frac{(x+1)^\gamma-(x-1)^\gamma}{x^\gamma} = \left(1+ \frac{1}{x}\right)^\gamma - \left(1- \frac{1}{x}\right)^\gamma.
\end{align*}
By the mean value theorem,
\[
\left(1+ \frac{1}{x}\right)^\gamma - \left(1- \frac{1}{x}\right)^\gamma \asymp x^{-1},
\qquad x \geq2.
\]
and therefore
\[
\frac{1}{x^\gamma}
\sum_{\substack{y\in\N_0\\ |x-y|=1}}
|x^\gamma-y^\gamma|
\asymp x^{-1}, \qquad x \geq 2.
\]
Since both $\alpha$ and $\alpha a_p$ belong to $(0,1]$, applying this
estimate with $\gamma=\alpha$ and $\gamma=\alpha a_p$ yields
\[
\mu_{p,\alpha}(x)\asymp x^{-1}, \qquad x \geq 2.
\]
The value at $x=1$ is positive and finite, so the above comparison extends
to all $x\in\N$. Consequently,
\[
\ell^p(\N,\mu_{p,\alpha})=\ell^p(\N,x^{-1}).
\]
Therefore, applying the abstract ground-state representation
\eqref{abstract ground state representation: one-function formulation}
to the present setting, we obtain, for every
$u\in \ell^p(\N,x^{-1})$ satisfying $u(0)=0$,
\begin{equation}\label{suboptimal ground state representation}
    \mathcal{E}_p^{\N_0}(u) = \sum_{x \in \N} w_{p, \alpha}(x)|u(x)|^p + \mathcal{R}_{p, \alpha}(u),
\end{equation}
where
\[
\begin{aligned}
\mathcal{R}_{p,\alpha}(u)
:=
\frac{1}{p}
\sum_{x\in\N}
\sum_{\substack{y\in\N\\ |x-y|=1}}
F_p\left(
    \frac{u(x)}{x^{\alpha q_p}},
    \frac{u(y)}{y^{\alpha q_p}}
\right)
x^{\alpha r_p}y^{\alpha q_p}.
\end{aligned}
\]
The explicit forms of $w_{p,\alpha}$ and $\mathcal R_{p,\alpha}$ follow from Remark \ref{rem: ground-state in terms of constants a_p}. In particular, since $F_p\geq 0$, we obtain the Hardy inequality
\begin{equation}\label{suboptimal Hardy on half-line}
\mathcal{E}_p^{\N_0}(u)
\geq
\sum_{x\in\N} w_{p,\alpha}(x)|u(x)|^p,
\end{equation}
for every $u\in\ell^p(\N,x^{-1})$ satisfying $u(0)=0$.
To extend this inequality to the larger class
\[
\ell^p(\N,x^{-2})\supsetneq \ell^p(\N,x^{-1}),
\]
we use the regularity Lemma \ref{lem: truncation Dirichlet half-line}. Indeed, let $u\in\ell^p(\N,x^{-2})$ satisfy $u(0)=0$, and let
$(u_j)$ be the sequence of finitely supported functions provided by
the lemma. Applying \eqref{suboptimal Hardy on half-line} to $u_j$ and letting $j \to \infty$ yields \eqref{e.HiSB}. This proves the first part of Theorem \ref{thm: critical Hardy weights for Dirichlet Laplacian}.

It remains to prove criticality. Let \(p>1\) and
\[
    \alpha\in\left(0,\frac{1}{a_p+1}\right].
\]
Suppose that \(w\) is a Hardy weight for \(\mathcal{E}_p^{\N_0}\) on \(\N\)
such that
\[
    w(x)\geq w_{p,\alpha}(x),\qquad x\in\N.
\]
Then the ground-state representation \eqref{suboptimal ground state representation}
gives
\begin{equation}\label{remainder inequality DL}
    0
    \leq
    \sum_{x\in\N}
    \bigl(w(x)-w_{p,\alpha}(x)\bigr)|u(x)|^p
    \leq
    \mathcal{R}_{p,\alpha}(u),
\end{equation}
for all \(u\in C_c(\N)\).  For
\[
    h_{p,\alpha}(x):=x^{\alpha q_p},\qquad x\in\N,
\]
the remainder vanishes:
\[
    \mathcal{R}_{p,\alpha}(h_{p,\alpha})=0.
\]
However, \(h_{p,\alpha}\) is not finitely supported. We therefore approximate
it by finitely supported functions for which the corresponding remainders tend
to zero.

For \(N\geq2\), define
\[
    u_N(x):=h_{p,\alpha}(x)\,\xi_N(x)^{\langle 2/p\rangle},
\]
where
\[
    \xi_N(x):=
    \begin{cases}
        1, & \text{if } x<N,\\[1mm]
        2-\dfrac{\log x}{\log N}, & \text{if } N\leq x<N^2,\\[2mm]
        0, & \text{if } N^2\leq x.
    \end{cases}
\]
Then \(u_N\in C_c(\N)\), and for every \(x\in\N\),
\[
    u_N(x)\to h_{p,\alpha}(x),
    \qquad N\to\infty.
\]
We claim that \(\mathcal{R}_{p,\alpha}(u_N)\to 0\) as \(N\to\infty\).
By the definition of \(\mathcal{R}_{p,\alpha}\) and
Lemma~\ref{lem: from Fp to F2}, the two directed nearest-neighbour
contributions corresponding to the edge \(\{x-1,x\}\), \(x\geq2\), are bounded by
\[
    2\Big(
        x^{\alpha r_p}(x-1)^{\alpha q_p}
        +(x-1)^{\alpha r_p}x^{\alpha q_p}
    \Big)
    \bigl(\xi_N(x)-\xi_N(x-1)\bigr)^2 .
\]
Since \(q_p,r_p\leq1\), \(q_p+r_p=1+a_p\), \(\alpha\leq1\), and
\(x/(x-1)\leq2\) for \(x\geq2\), this is bounded by
\[
    8\bigl(\xi_N(x)-\xi_N(x-1)\bigr)^2
    (x-1)^{\alpha(1+a_p)}.
\]
Therefore
\begin{align*}
    \mathcal{R}_{p,\alpha}(u_N)
    &\leq
    8 \sum_{x \geq 2}
    \bigl(\xi_N(x)-\xi_N(x-1)\big)^2
    (x-1)^{\alpha(1+a_p)}
    \\
    &=
    \frac{8}{\log^2 N}
    \sum_{x=N+1}^{N^2}
    \log^2\left(\frac{x}{x-1}\right)
    (x-1)^{\alpha(1+a_p)}
    \\
    &\leq
    \frac{8}{\log^2 N}
    \sum_{x=N+1}^{N^2}
    (x-1)^{\alpha(1+a_p)-2}.
\end{align*}
Since \(\alpha(1+a_p)\leq1\), we get
\[
    \mathcal{R}_{p,\alpha}(u_N)
    \leq
    \frac{8}{\log^2 N}
    \sum_{x=N+1}^{N^2}\frac{1}{x-1}
    \leq
    \frac{8\log(N+1)}{\log^2 N}
    \longrightarrow0 .
\]
Combining this estimate with \eqref{remainder inequality DL} and Fatou's lemma, we obtain
\begin{align*}
     0 \leq \sum_{x \in \N} \big(w(x)-w_{p, \alpha}(x)\big)|h_{p, \alpha}(x)|^p  &\leq \liminf_{N \rightarrow \infty}\sum_{x \in \N} \big(w(x)-w_{p, \alpha}(x)\big)|u_N(x)|^p \\
     &\leq \liminf_{N \rightarrow \infty} \mathcal{R}_{p, \alpha}(u_N) = 0.
\end{align*}
Since the summand is non-negative and $h_{p, \alpha}(x) > 0$, we conclude that 
$$
w(x) = w_{p, \alpha}(x), \qquad x \in \N.
$$
\end{proof}
Next, by analyzing the Hardy weights $w_{p,\alpha}$ more closely, we obtain a $p$-analogue of the classical Hardy inequality \eqref{classical quadratic Hardy} with the sharp constant.

\begin{proof}[Proof of Theorem \ref{thm: sharp inverse-square Hardy}]
Let $\alpha \in (0, 1]$. By the binomial expansion, 
$$
w_{p, \alpha}(x) = -2\sum_{i=1}^\infty \left(\frac{1}{a_p+1} {\alpha a_p \choose 2i} + \frac{a_p}{a_p+1} {\alpha \choose 2i}\right)x^{-2i}, \qquad x \geq 2.
$$
Since \(\alpha,\alpha a_p\in(0,1]\), all terms in the series above are
nonnegative. Consequently,
$$
w_{p, \alpha} (x) \geq c_{p, \alpha} x^{-2}, \qquad \forall x \geq 2,
$$
where
$$
c_{p, \alpha} := -2 \left(\frac{1}{a_p+1} {\alpha a_p \choose 2} + \frac{a_p}{a_p+1} {\alpha \choose 2}\right) = \frac{\alpha a_p}{a_p+1} \bigl(2-\alpha(a_p+1)\bigr).
$$
As a function of $\alpha$, the constant is maximized at $\alpha_*= 1/(a_p+1)$, where
$$
c_{p, \alpha_*} = \frac{a_p}{(a_p+1)^2} = \frac{p-1}{p^2}.
$$
This proves
$$w_{p, \alpha_*}(x) \geq \frac{(p-1)}{p^2} x^{-2}, \qquad \forall x \geq 2.
$$
It remains to verify that $$w_{p, \alpha_*}(1) \geq \frac{p-1}{p^2}.$$ 
Set
$$\beta_*:=\alpha_*a_p=\frac{a_p}{1+a_p}.$$
Then $\alpha_*+\beta_*=1$, and
$$w_{p,\alpha_*}(1) = \alpha_*\bigl(2-2^{\beta_*}\bigr) + \beta_*\bigl(2-2^{\alpha_*}\bigr).$$
By the convexity of $t\mapsto 2^t$, for every $t\in[0,1]$,
$$2^t\leq (1-t)2^0+t2^1=1+t.$$
Therefore,
$$2-2^{\beta_*}\geq 1-\beta_*=\alpha_*, \qquad 2-2^{\alpha_*}\geq 1-\alpha_*=\beta_*,$$
which implies the desired estimate at $x=1$,
$$w_{p,\alpha_*}(1) \geq \alpha_*^2+\beta_*^2 \geq \alpha_*\beta_* = \frac{a_p}{(1+a_p)^2} = \frac{p-1}{p^2}.
$$
This proves the Hardy inequality \eqref{e.sharp inverse-square Hardy} with the classical inverse-square weight. We next prove the sharpness of the constant.

Let $0 < \beta < 1/p$, and for each $N \in \N$ we define a finitely supported function $u_N : \N_0 \rightarrow \R$ as
\begin{align*}
    u_N(x) := 
    \begin{cases}
        x^\beta, &\textrm{if $ x < N$}, \\
        N^\beta \left(2- \frac{x}{N}\right), &\textrm{if $ N \leq x < 2N$}, \\
        0, &\textrm{if $2N \leq x$}.
    \end{cases}
\end{align*}
On the one hand, we have 
\begin{equation}
 \sum_{x \in \N} |u_N(x)|^p x^{-2} \geq \sum_{x =1}^N x^{p \beta-2}.
\end{equation}
On the other hand, by \eqref{Dirichlet Laplacian p-form_2},
\begin{align*}
 \mathcal{E}_p^{\N_0}(u_N) &= \sum_{x=2}^N \big(x^\beta - (x-1)^\beta \big) \big(x^{\beta(p-1)} - (x-1)^{\beta (p-1)}\big) + E_N + 1 \\
 & \leq \beta^2 (p-1) \sum_{x=1}^{N-1} x^{p \beta -2} + E_N + 1,
\end{align*}
where
\begin{align*}
 E_N &:= \sum_{x=N+1}^{2N} N^{p(\beta-1)} \Big( (2N - (x-1))^{p-1} - (2N - x)^{p-1} \Big) = N^{p \beta-1}.
\end{align*}
Therefore \emph{any} constant $C$ in \eqref{e.sharp inverse-square Hardy} satisfies
$$
C \sum_{x=1}^N x^{p \beta-2} \leq \beta^2 (p-1) \sum_{x=1}^N x^{p \beta-2} + N^{p \beta-1} + 1.
$$
Since $\beta < 1/p$, letting $N \rightarrow \infty$ gives
\begin{align*}
 C \leq \beta^2 (p-1) + \frac{1}{S_\beta},
\end{align*}
where 
$$
S_\beta  = \sum_{x \in \N} x^{p\beta-2}.
$$
As $\beta \rightarrow 1/p$, $S_\beta \rightarrow \infty$, and hence
$$
C \leq \frac{p-1}{p^2},
$$
proving the sharpness.     
\end{proof}
\begin{remark}\label{rem: inverse-square Hardy by comparison}
The inequality \eqref{e.sharp inverse-square Hardy} may also be obtained
directly from \eqref{equivalence of 2 and p form} and the classical quadratic
discrete Hardy inequality \eqref{classical quadratic Hardy}. Indeed, let
\(u\in\ell^p(\N)\) and put
\[
    v(x):=|u(x)|^{p/2},\qquad x\in\N_0,
\]
with \(v(0)=0\). Then \eqref{equivalence of 2 and p form} gives
\[
    \mathcal{E}_p^{\N_0}(u)
    \geq
    \frac{4(p-1)}{p^2}
    \sum_{x=1}^{\infty}
    \bigl(v(x)-v(x-1)\bigr)^2.
\]
Applying \eqref{classical quadratic Hardy} to \(v\), we obtain
\[
    \sum_{x=1}^{\infty}
    \bigl(v(x)-v(x-1)\bigr)^2
    \geq
    \frac14
    \sum_{x=1}^{\infty}
    \frac{v(x)^2}{x^2}.
\]
Since \(v(x)^2=|u(x)|^p\), this yields
\[
    \mathcal{E}_p^{\N_0}(u)
    \geq
    \frac{p-1}{p^2}
    \sum_{x=1}^{\infty}
    \frac{|u(x)|^p}{x^2}.
\]
This comparison recovers the constant in
\eqref{e.sharp inverse-square Hardy}. It does not, by itself, prove sharpness,
and in more general situations such a comparison need not yield the optimal
Hardy weight.
\end{remark}
\begin{corollary}\label{cor: Dirichlet contractivity}
If \(p\in(1,\infty)\), then \((p-1)/p^2\) is the largest constant \(C\geq0\)
such that \(-\Delta+C M_{x^{-2}}\) generates a contraction semigroup
on \(\ell^p(\N)\).
\end{corollary}
\begin{proof}
Let
\[
    w_C(x):=Cx^{-2}-V_{\partial}(x),
    \qquad x\in\N.
\]
Since
\(
    \Delta=L_\kappa+M_{V_{\partial}},
\)
we have
\[
    -L_\kappa+M_{w_C}
    =
    -\Delta+C M_{x^{-2}}.
\]
Moreover, by \eqref{e.rbfN0}, the condition\/ \textup{(1)} of
Theorem~\ref{thm: contractivity of perturbed semigroup} is equivalent to
\[
    \mathcal E_p^{\N_0}(u)
    \geq
    C\sum_{x\in\N}\frac{|u(x)|^p}{x^2},
    \qquad u\in\ell^p(\N).
\]
The conclusion now follows from
Theorem~\ref{thm: sharp inverse-square Hardy} and the sharpness of its
constant.
\end{proof}

%\begin{remark}\label{rem: inverse-square Hardy by comparison}
%The inequality \eqref{e.sharp inverse-square Hardy} may also be obtained
%directly from the classical quadratic discrete Hardy inequality. Indeed, for
%\(a,b\in\R\),
%[
%    (a-b)
%    \bigl(
%        a^{\langle p-1\rangle}
%        -
%        b^{\langle p-1\rangle}
%    )
%    \geq
%    \frac{4(p-1)}{p^2}
%    \bigl(
%        |a|^{p/2}
%        -
%        |b|^{p/2}
%    \bigr)^2 .
%\]
%Applying this with \(a=u(x)\), \(b=u(x-1)\), and then using the classical
%quadratic discrete Hardy inequality for \(v(x):=|u(x)|^{p/2}\), we get
%\[
%\begin{aligned}
%    \mathcal{E}_p^{\N_0}(u)
%    &\geq
%    \frac{4(p-1)}{p^2}
%    \sum_{x=1}^{\infty}
%    \bigl(v(x)-v(x-1)\bigr)^2  \\
%    &\geq
%    \frac{4(p-1)}{p^2}\cdot\frac14
%    \sum_{x=1}^{\infty}
%    \frac{v(x)^2}{x^2}  \\
%    &=
%    \frac{p-1}{p^2}
%    \sum_{x=1}^{\infty}
%    \frac{|u(x)|^p}{x^2}.
%\end{aligned}
%\]
%This comparison explains the constant in
%\eqref{e.sharp inverse-square Hardy}. It does not, by itself, prove sharpness,
%and in more general situations such a %comparison need not yield the optimal
%Hardy weight.
%\end{remark}

\section{Fractional Laplacian on integers}
\label{sec: fractional Laplacian}

Throughout this section, \(\Delta\) denotes the standard positive discrete
Laplacian on \(\Z\),
\[
    \Delta u(x):=2u(x)-u(x-1)-u(x+1),
    \qquad x\in\Z,
\]
and \(\Delta^\sigma\) denotes its fractional power defined in
Subsection~\ref{subsec: fractional Laplacian}.

Below, we specialize the abstract framework of
Section~\ref{sec: abstract framework} to \(\Delta^\sigma\) and prove
Theorem~\ref{thm: critical Hardy weights for fractional Laplacian}. Accordingly, we take
\[
    X=\Omega=\Z.
\]

\subsection{Superharmonic functions}

For \(\sigma\in(0,1)\), we use the fractional kernel
\[
    k(x,y):=k_\sigma(x-y),\qquad x,y\in\Z,
\]
where
\[
    k_\sigma(0)=0,
\]
and, for \(z\neq0\),
\[
    k_\sigma(z)
    =
    \mathcal{A}_{-\sigma}
    \frac{\Gamma(|z|-\sigma)}{\Gamma(|z|+1+\sigma)},
\]
where \(\mathcal{A}_{-\sigma}\) is defined in
\eqref{e.leading-constant-ksigma}.
As \(|z|\to\infty\),
\begin{equation}\label{e.ksigma-asymptotic}
    k_\sigma(z)
    =
    \mathcal{A}_{-\sigma}\frac{1}{|z|^{1+2\sigma}}
    +
    O\left(\frac{1}{|z|^{2+2\sigma}}\right).
\end{equation}

We shall also use the same gamma-function expression for negative parameters.
More precisely, for \(\theta\in(-1/2,0)\) we define
\[
    k_\theta(z)
    :=
    \mathcal{A}_{-\theta}
    \frac{\Gamma(|z|-\theta)}
         {\Gamma(|z|+1+\theta)},
    \qquad z\in\Z.
\]
In this case \(k_\theta\) is strictly positive on \(\Z\).

Recall that, for the kernel \(k_\sigma\), the associated operator is defined
pointwise on
\[
    \mathcal F_\sigma(\Z)
    :=
    \left\{
        u:\Z\to\R:
        \sum_{y\in\Z} k_\sigma(x-y)|u(y)|<\infty
        \text{ for every }x\in\Z
    \right\}.
\]
Since \(k_\sigma\in\ell^1(\Z)\subset\ell^2(\Z)\), every
\(u\in\ell^2(\Z)\) belongs to \(\mathcal F_\sigma(\Z)\).
On \(\ell^2(\Z)\), this operator agrees with the positive fractional Laplacian
\(\Delta^\sigma\) defined by the spectral calculus:
\[
    \Delta^\sigma u(x)
    =
    \sum_{y\in\Z}
    \bigl(u(x)-u(y)\bigr)k_\sigma(x-y),
    \qquad x\in\Z.
\]
Throughout this section, we use the same notation \(\Delta^\sigma\) for the
pointwise extension of this expression to \(\mathcal F_\sigma(\Z)\).
The corresponding Sobolev--Bregman form is
\(\mathcal{E}_{p}^{\Z,\sigma}\), defined in
\eqref{def: fractional SB form}.
The proof of
Theorem~\ref{thm: critical Hardy weights for fractional Laplacian} follows the
same general strategy as in the Dirichlet case, but the estimates are more
delicate because the fractional kernel is nonlocal. 
Proposition~\ref{prop: abstract ground state representation: one-function formulation} and the superharmonic functions
\[
    g=k_{-\alpha},\qquad \alpha\in(\sigma,1/2),
\]
 yield Hardy inequalities
\eqref{fractional p Hardy inequality} with the weights \(w_{p,\sigma,\alpha}\) given by \eqref{fractional p Hardy weight}.

The rest of the proof has two parts. First, we identify the admissible weighted
space and determine the asymptotic behaviour of the Hardy weights. Second, we
prove criticality by approximating the corresponding ground state with
finitely supported functions. In the nonlocal case, the remainder contains
interactions between distant points of \(\Z\), so this approximation requires
additional cutoff estimates.

Henceforth, we fix \(\sigma\in(0,1/2)\) and
\(\alpha\in(\sigma,1/2)\). The starting point is the family of positive
functions \(k_{-\alpha}\). These functions belong to
\(\mathcal F_\sigma(\Z)\) and are strictly superharmonic for
\(\Delta^\sigma\). This was first established in \cite{MR3882021}; a different
proof was later given in \cite{MR4612316}.

\begin{lemma}\label{lem: fractional superharmonic functions}
Let \(\sigma\in(0,1/2)\) and \(\alpha\in(\sigma,1/2)\). Then
\begin{equation}
    \Delta^\sigma k_{-\alpha}=k_{\sigma-\alpha}.
\end{equation}
In particular, \(\Delta^\sigma k_{-\alpha}>0\), and hence \(k_{-\alpha}\) is
strictly superharmonic.
\end{lemma}

\subsection{Hardy weights}
\label{subsec: fractional Hardy weights}

We first apply the one-function ground-state representation to the
superharmonic functions \(k_{-\alpha}\).

\begin{proposition}[Fractional Hardy weights]
\label{prop: fractional Hardy weights}
Let \(p\in(1,\infty)\), \(\sigma\in(0,1/2)\), and
\(\alpha\in(\sigma,1/2)\). Put
\[
    h_{p,\sigma,\alpha}:=k_{-\alpha}^{q_p}
\]
and
\[
    \mu_{p,\sigma,\alpha}:=\nu_{k_{-\alpha}},
\]
where \(\nu_g\) is defined in
\eqref{abstract one-function reference measure}. Define
\begin{equation}\label{fractional Hardy weight}
    w_{p,\sigma,\alpha}(x)
    :=
    \frac{1}{a_p+1}
    \frac{\Delta^\sigma k_{-\alpha}^{a_p}(x)}
         {k_{-\alpha}^{a_p}(x)}
    +
    \frac{a_p}{a_p+1}
    \frac{\Delta^\sigma k_{-\alpha}(x)}
         {k_{-\alpha}(x)},
    \qquad x\in\Z.
\end{equation}
Then
\[
    w_{p,\sigma,\alpha}(x)>0,
    \qquad x\in\Z,
\]
and, for every \(u\in\ell^p(\Z,\mu_{p,\sigma,\alpha})\),
\begin{equation}\label{fractional ground state representation}
    \mathcal E_p^{\Z,\sigma}(u)
    =
    \sum_{x\in\Z} w_{p,\sigma,\alpha}(x)|u(x)|^p
    +
    \mathcal R_{p,\sigma,\alpha}(u),
\end{equation}
where
\begin{align}
    \mathcal R_{p,\sigma,\alpha}(u)
    :={}&
    \frac{1}{p}
    \sum_{x,y\in\Z}
    F_p\left(
        \frac{u(x)}{k_{-\alpha}(x)^{q_p}},
        \frac{u(y)}{k_{-\alpha}(y)^{q_p}}
    \right)
    \notag\\
    &\qquad\times
    k_{-\alpha}(x)^{r_p}
    k_{-\alpha}(y)^{q_p}
    k_\sigma(x-y).
    \label{fractional ground state remainder}
\end{align}
In particular,
\begin{equation}\label{fractional Hardy inequality with reference measure}
    \mathcal E_p^{\Z,\sigma}(u)
    \geq
    \sum_{x\in\Z} w_{p,\sigma,\alpha}(x)|u(x)|^p .
\end{equation}
\end{proposition}

\begin{proof}
By Lemma~\ref{lem: fractional superharmonic functions},
\(k_{-\alpha}\) is strictly positive and strictly superharmonic on \(\Z\).
Therefore
Proposition~\ref{prop: abstract ground state representation: one-function formulation}
applies with \(g=k_{-\alpha}\). The identity
\eqref{fractional ground state representation}, the formula
\eqref{fractional Hardy weight}, and the remainder
\eqref{fractional ground state remainder} follow from this proposition.

The positivity of \(w_{p,\sigma,\alpha}\) follows from
\[
    \Delta^\sigma k_{-\alpha}=k_{\sigma-\alpha}>0
\]
and the superharmonicity of \(k_{-\alpha}^{a_p}\).
\end{proof}

\subsection{Reference measure and asymptotics}
\label{subsec: fractional reference measure and asymptotics}

To identify the admissible space \(\ell^p(\Z,\mu_{p,\sigma,\alpha})\) and
determine the asymptotic behaviour of \(w_{p,\sigma,\alpha}\), we first
establish some auxiliary estimates.

\begin{lemma}[Convolution estimate]
\label{lem: fractional convolution estimate}
Let $\sigma\in(0,1/2)$ and $a \in (0,1]$. Then
\begin{equation}\label{fractional convolution estimate}
\sum_{y\in\Z\setminus\{0,x\}}
\frac{1}{|y|^{1+a}}
\frac{1}{|x-y|^{1+2\sigma}}
\lesssim
|x|^{-1-\min\{2\sigma,a\}},
\qquad x \neq 0.
\end{equation}
\end{lemma}

\begin{proof}
By symmetry it is enough to consider $x>0$. Split the sum into
the regions
\[
y\leq \frac{x}{2},
\qquad
\frac{x}{2}<y<\frac{3x}{2},
\qquad
y\geq\frac{3x}{2},
\]
and denote the corresponding sums by $S_1,S_2,S_3$ respectively. On the first region, using $|x-y| \geq x/2$,  we get
\[
S_1
\lesssim
x^{-1-2\sigma}
\sum_{\substack{y\leq x/2\\y\neq0}}
|y|^{-1-a}
\lesssim
x^{-1-2\sigma}.
\]
On the second region, $|y|\asymp x$, and hence
\[
S_2
\lesssim
x^{-1-a}
\sum_{y\neq x}|x-y|^{-1-2\sigma}
\lesssim
x^{-1-a}.
\]
Finally, using $|x-y| \geq x/2$ in the third region, we have
\[
S_3
\lesssim
x^{-1-2\sigma}
\sum_{y\geq3x/2}y^{-1-a}
\lesssim
x^{-1-2\sigma-a}.
\]
This proves \eqref{fractional convolution estimate}.
\end{proof}

\begin{lemma}\label{lem:asymptotics of kalpha}
Let \(\sigma\in(0,1/2)\), \(\alpha\in(\sigma,1/2)\), and
\(a\in(0,1]\). Put
\[
    \gamma=\frac{1-a(1-2\alpha)}{2}.
\]
Then \(\gamma\in[\alpha,1/2)\), and, as \(|x|\to\infty\),
\begin{equation}\label{asymptotics of k^a}
    \frac{1}{\mathcal A_\alpha^a}k_{-\alpha}(x)^a
    =
    \frac{1}{\mathcal A_\gamma}k_{-\gamma}(x)
    +
    O\left(|x|^{2\gamma-2}\right),
\end{equation}
and
\begin{equation}\label{asymptotics of Lk^a}
    \frac{1}{\mathcal A_\alpha^a}
    \Delta^\sigma k_{-\alpha}^a(x)
    =
    \frac{1}{\mathcal A_\gamma}
    \Delta^\sigma k_{-\gamma}(x)
    +
    O\left(
        |x|^{-1-\min\{2\sigma,1-2\gamma\}}
    \right).
\end{equation}
\end{lemma}

\begin{proof}
Similarly to \eqref{e.ksigma-asymptotic},
\begin{equation}\label{asymptotics of k}
    k_{-\alpha}(x)
    =
    \mathcal A_\alpha |x|^{2\alpha-1}
    \left(1+O(|x|^{-1})\right),
    \qquad |x|\to\infty.
\end{equation}
Since \(a\in(0,1]\), this gives
\[
    \frac{1}{\mathcal A_\alpha^a}k_{-\alpha}(x)^a
    =
    |x|^{a(2\alpha-1)}
    \left(1+O(|x|^{-1})\right).
\]
By the definition of \(\gamma\),
\[
    a(2\alpha-1)=2\gamma-1.
\]
Thus
\[
    \frac{1}{\mathcal A_\alpha^a}k_{-\alpha}(x)^a
    =
    |x|^{2\gamma-1}
    +
    O\left(|x|^{2\gamma-2}\right).
\]
On the other hand,
\[
    \frac{1}{\mathcal A_\gamma}k_{-\gamma}(x)
    =
    |x|^{2\gamma-1}
    +
    O\left(|x|^{2\gamma-2}\right).
\]
This proves \eqref{asymptotics of k^a}.

Let
\[
    r(x)
    :=
    \frac{1}{\mathcal A_\alpha^a}k_{-\alpha}(x)^a
    -
    \frac{1}{\mathcal A_\gamma}k_{-\gamma}(x).
\]
Then
\[
    r(x)=O\left(|x|^{2\gamma-2}\right),
    \qquad |x|\to\infty.
\]
By Lemma~\ref{lem: fractional convolution estimate}, 
\[
    \Delta^\sigma r(x)
    =
    O\left(
        |x|^{-1-\min\{2\sigma,1-2\gamma\}}
    \right).
\]
Since
\[
    \frac{1}{\mathcal A_\alpha^a}k_{-\alpha}^a
    =
    \frac{1}{\mathcal A_\gamma}k_{-\gamma}+r,
\]
applying \(\Delta^\sigma\) gives \eqref{asymptotics of Lk^a}.
\end{proof}

\begin{lemma}\label{lem: power difference estimate}
Let \(\sigma\in(0,1/2)\) and \(\eta\in(0,1/2)\). Then, for \(x\geq2\),
\[
    \sum_{y\in\Z\setminus\{0,x\}}
    \left|x^{2\eta-1}-|y|^{2\eta-1}\right|
    |x-y|^{-1-2\sigma}
    \lesssim
    x^{2\eta-2\sigma-1}.
\]
\end{lemma}

\begin{proof}
We split the sum into the three regions
\[
    0<|y|\leq \frac{x}{2},
    \qquad
    \frac{x}{2}<|y|<2x,
    \qquad
    |y|\geq2x.
\]

If \(0<|y|\leq x/2\), then, since \(2\eta-1<0\),
\[
    \left|x^{2\eta-1}-|y|^{2\eta-1}\right|
    \lesssim |y|^{2\eta-1}.
\]
Moreover, \(|x-y|\geq x/2\), and hence
\[
\sum_{0<|y|\leq x/2}
\left|x^{2\eta-1}-|y|^{2\eta-1}\right|
|x-y|^{-1-2\sigma}
\lesssim
x^{-1-2\sigma}
\sum_{1\leq |y|\leq x/2}|y|^{2\eta-1}
\lesssim
x^{2\eta-2\sigma-1}.
\]

Next, suppose that \(x/2<|y|<2x\). If \(y>0\), then the mean value theorem gives
\[
    \left|x^{2\eta-1}-y^{2\eta-1}\right|
    \lesssim
    x^{2\eta-2}|x-y|.
\]
Therefore
\[
\sum_{\substack{x/2<y<2x\\ y\neq x}}
\left|x^{2\eta-1}-y^{2\eta-1}\right|
|x-y|^{-1-2\sigma}
\lesssim
x^{2\eta-2}
\sum_{1\leq |x-y|\leq x}|x-y|^{-2\sigma}
\lesssim
x^{2\eta-2\sigma-1}.
\]
If \(y<0\), then \(|x-y|\asymp x\) and
\[
    \left|x^{2\eta-1}-|y|^{2\eta-1}\right|
    \lesssim x^{2\eta-1}.
\]
Thus
\[
\sum_{-2x<y<-x/2}
\left|x^{2\eta-1}-|y|^{2\eta-1}\right|
|x-y|^{-1-2\sigma}
\lesssim
x\cdot x^{2\eta-1}\cdot x^{-1-2\sigma}
=
x^{2\eta-2\sigma-1}.
\]

Finally, if \(|y|\geq2x\), then \(|x-y|\asymp |y|\) and
\[
    \left|x^{2\eta-1}-|y|^{2\eta-1}\right|
    \lesssim x^{2\eta-1}.
\]
Hence
\[
\sum_{|y|\geq2x}
\left|x^{2\eta-1}-|y|^{2\eta-1}\right|
|x-y|^{-1-2\sigma}
\lesssim
x^{2\eta-1}
\sum_{|y|\geq2x}|y|^{-1-2\sigma}
\lesssim
x^{2\eta-2\sigma-1}.
\]
Combining the three estimates proves the claim.
\end{proof}

\begin{lemma}[Reference measure estimate]
\label{lem: fractional reference measure estimate}
Let \(\sigma\in(0,1/2)\), \(\alpha\in(\sigma,1/2)\), and
\(a\in(0,1]\). Then
\[
    \frac{1}{k_{-\alpha}(x)^a}
    \sum_{y\in\Z}
    \left|
        k_{-\alpha}(x)^a-k_{-\alpha}(y)^a
    \right|
    k_\sigma(x-y)
    \asymp
    (1+|x|)^{-2\sigma},
    \qquad x\in\Z.
\]
\end{lemma}

\begin{proof}
Put
\[
    \beta:=a(1-2\alpha).
\]
Then \(\beta\in(0,1)\). By \eqref{asymptotics of k},
\begin{equation}\label{estimates of k^a}
    k_{-\alpha}(x)^a\asymp (1+|x|)^{-\beta},
    \qquad x\in\Z.
\end{equation}

We first prove the upper bound. It is enough to consider \(|x|\geq2\); the
remaining values of \(x\) are harmless. By symmetry, assume \(x\geq2\). Using \eqref{estimates of k^a}, the estimate for the remainder in
Lemma~\ref{lem:asymptotics of kalpha}, and
Lemma~\ref{lem: power difference estimate}, with
\[
    2\eta-1=-\beta,
\]
we obtain
\[
\begin{aligned}
    \sum_{y\in\Z}
    \left|
        k_{-\alpha}(x)^a-k_{-\alpha}(y)^a
    \right|
    k_\sigma(x-y)
    &\lesssim
    x^{-\beta-2\sigma}.
\end{aligned}
\]
Since \(k_{-\alpha}(x)^a\asymp x^{-\beta}\), this gives
\[
    \frac{1}{k_{-\alpha}(x)^a}
    \sum_{y\in\Z}
    \left|
        k_{-\alpha}(x)^a-k_{-\alpha}(y)^a
    \right|
    k_\sigma(x-y)
    \lesssim
    x^{-2\sigma}.
\]

For the lower bound, assume again that \(x\) is large and positive. If
\(x/2\leq y\leq 3x/4\), then \eqref{estimates of k^a} implies
\[
    \left|
        k_{-\alpha}(x)^a-k_{-\alpha}(y)^a
    \right|
    \gtrsim x^{-\beta}.
\]
Moreover, \(|x-y|\asymp x\), and hence
\[
    k_\sigma(x-y)\gtrsim x^{-1-2\sigma}.
\]
There are \(\asymp x\) such values of \(y\), so
\[
    \sum_{y\in\Z}
    \left|
        k_{-\alpha}(x)^a-k_{-\alpha}(y)^a
    \right|
    k_\sigma(x-y)
    \gtrsim
    x\cdot x^{-\beta}\cdot x^{-1-2\sigma}
    =
    x^{-\beta-2\sigma}.
\]
Dividing by \(k_{-\alpha}(x)^a\asymp x^{-\beta}\), we obtain the lower bound
\(x^{-2\sigma}\). By symmetry, the same argument applies for negative \(x\).
The finitely many remaining values of \(x\) are absorbed into the constants.
This proves the claim.
\end{proof}

\begin{lemma}[Quotient asymptotics]
\label{lem: fractional quotient asymptotics}
Let \(\sigma\in(0,1/2)\), \(\alpha\in(\sigma,1/2)\), and
\(a\in(0,1]\). Define \(\gamma\) by
\[
    1-2\gamma=a(1-2\alpha).
\]
Then \(\gamma\in(\sigma,1/2)\), and, as \(|x|\to\infty\),
\[
    \frac{\Delta^\sigma k_{-\alpha}^{a}(x)}
         {k_{-\alpha}^{a}(x)}
    =
    \frac{\mathcal A_{\gamma-\sigma}}{\mathcal A_\gamma}
    \frac{1}{|x|^{2\sigma}}
    +
    O\left(
        \frac{1}
        {|x|^{2\sigma+1-\max\{2\sigma,1-2\gamma\}}}
    \right).
\]
\end{lemma}

\begin{proof}
The inclusion \(\gamma\in(\sigma,1/2)\) follows from
\(a\in(0,1]\) and \(\alpha\in(\sigma,1/2)\). By
Lemma~\ref{lem:asymptotics of kalpha},
\[
    k_{-\alpha}(x)^a
    =
    \mathcal A_\alpha^a |x|^{-1+2\gamma}
    +
    O\left(|x|^{-2+2\gamma}\right),
\]
and
\[
    \Delta^\sigma k_{-\alpha}^{a}(x)
    =
    \frac{\mathcal A_\alpha^a}{\mathcal A_\gamma}
    \mathcal A_{\gamma-\sigma}
    |x|^{-1-2\sigma+2\gamma}
    +
    O\left(
        |x|^{-1-\min\{2\sigma,1-2\gamma\}}
    \right).
\]
Dividing the second asymptotic formula by the first one gives the claim.
\end{proof}

\begin{proposition}[Reference measure and asymptotics]
\label{prop: fractional reference measure and asymptotics}
Let \(p\in(1,\infty)\), \(\sigma\in(0,1/2)\), and
\(\alpha\in(\sigma,1/2)\). Let \(\mu_{p,\sigma,\alpha}\) and
\(w_{p,\sigma,\alpha}\) be defined as in
Proposition~\ref{prop: fractional Hardy weights}. Define \(\gamma_p\) by
\[
    1-2\gamma_p=a_p(1-2\alpha).
\]
Then
\[
    \mu_{p,\sigma,\alpha}(x)\asymp (1+|x|)^{-2\sigma},
    \qquad x\in\Z.
\]
Moreover, as \(|x|\to\infty\),
\[
    w_{p,\sigma,\alpha}(x)
    =
    \frac{\Psi(p,\sigma,\alpha)}{|x|^{2\sigma}}
    +
    O\left(
        \frac{1}
        {|x|^{2\sigma+1-\max\{2\sigma,1-2\alpha\}}}
    \right),
\]
where
\[
    \Psi(p,\sigma,\alpha)
    =
    \frac{1}{a_p+1}
    \frac{\mathcal A_{\gamma_p-\sigma}}{\mathcal A_{\gamma_p}}
    +
    \frac{a_p}{a_p+1}
    \frac{\mathcal A_{\alpha-\sigma}}{\mathcal A_{\alpha}}.
\]
\end{proposition}

\begin{proof}
We first identify the reference measure. By the definition of
\(\mu_{p,\sigma,\alpha}\) and \eqref{abstract one-function reference measure},
\[
\begin{aligned}
    \mu_{p,\sigma,\alpha}(x)
    &=
    \frac{1}{a_p+1}
    \frac{1}{k_{-\alpha}(x)^{a_p}}
    \sum_{y\in\Z}
    \left|
        k_{-\alpha}(x)^{a_p}
        -
        k_{-\alpha}(y)^{a_p}
    \right|
    k_\sigma(x-y)
    \\
    &\quad
    +
    \frac{a_p}{a_p+1}
    \frac{1}{k_{-\alpha}(x)}
    \sum_{y\in\Z}
    \left|
        k_{-\alpha}(x)-k_{-\alpha}(y)
    \right|
    k_\sigma(x-y).
\end{aligned}
\]
Applying Lemma~\ref{lem: fractional reference measure estimate}, first with
\(a=a_p\) and then with \(a=1\), gives
\[
    \mu_{p,\sigma,\alpha}(x)
    \asymp
    (1+|x|)^{-2\sigma},
    \qquad x\in\Z.
\]

It remains to prove the asymptotic formula for \(w_{p,\sigma,\alpha}\).
Since \(a_p\in(0,1]\) and \(\alpha\in(\sigma,1/2)\), the number
\(\gamma_p\) satisfies
\[
    \sigma<\alpha\leq \gamma_p<\frac12.
\]
By Lemma~\ref{lem: fractional quotient asymptotics}, applied with \(a=1\),
we have
\[
    \frac{\Delta^\sigma k_{-\alpha}(x)}
         {k_{-\alpha}(x)}
    =
    \frac{\mathcal A_{\alpha-\sigma}}{\mathcal A_{\alpha}}
    \frac{1}{|x|^{2\sigma}}
    +
    O\left(
        \frac{1}
        {|x|^{2\sigma+1-\max\{2\sigma,1-2\alpha\}}}
    \right).
\]
Applying the same lemma with \(a=a_p\), so that the corresponding parameter is
\(\gamma_p\), gives
\[
    \frac{\Delta^\sigma k_{-\alpha}^{a_p}(x)}
         {k_{-\alpha}^{a_p}(x)}
    =
    \frac{\mathcal A_{\gamma_p-\sigma}}{\mathcal A_{\gamma_p}}
    \frac{1}{|x|^{2\sigma}}
    +
    O\left(
        \frac{1}
        {|x|^{2\sigma+1-\max\{2\sigma,1-2\gamma_p\}}}
    \right).
\]
Since
\[
    1-2\gamma_p=a_p(1-2\alpha)\leq 1-2\alpha,
\]
we have
\[
    \max\{2\sigma,1-2\gamma_p\}
    \leq
    \max\{2\sigma,1-2\alpha\}.
\]
Therefore the second error term is dominated by
\[
    O\left(
        \frac{1}
        {|x|^{2\sigma+1-\max\{2\sigma,1-2\alpha\}}}
    \right).
\]
Combining the last two asymptotic formulae with
\eqref{fractional Hardy weight} gives the asserted asymptotic formula for
\(w_{p,\sigma,\alpha}\).
\end{proof}

\subsection{Criticality}
\label{subsec: fractional criticality}
In this subsection, we determine the range of the parameter 
$\alpha$ for which the Hardy weights obtained in Proposition \ref{prop: fractional Hardy weights} are critical. To do so, we first establish some preliminary statements. Henceforth, we assume $p \in (1, \infty)$, $\sigma \in (0, 1/2)$, and $\alpha \in (\sigma, 1/2)$.

Let $v : \Z \rightarrow \R$ and $ \mathcal{M}_{p, \sigma, \alpha}(v)$ be the form given by
\begin{equation}\label{M form}
    \mathcal{M}_{p, \sigma, \alpha}(v) := \sum_{x \in \Z}\sum_{y \in \Z} \left(v(x) - v(y)\right)^2 k_{-\alpha}(x)^{r_p}k_{-\alpha}(y)^{q_p}
    k_\sigma (x-y).
\end{equation}
The next result shows that boundedness alone is sufficient to deduce convergence of the form \eqref{M form}. This is well-known for $p=2$ \cite{MR4612316}, we extend it to $p > 1$.

\begin{lemma}\label{lem: boundedness implies convergence to zero}
Let $p \in (1, \infty)$ and $N \in \N$. Let $0 \leq \xi_N(x) \leq 1$ be a sequence of finitely supported functions on $\Z$, such that 
$$\xi_N (x)  \rightarrow 1, \quad x \in \Z, \quad \operatorname{and} \quad \sup_{N \in \N} \mathcal{M}_{p, \sigma, \alpha}(\xi_N) < \infty.$$ 
Then there exists a finitely supported sequence $0 \leq e_N (x) \leq 1$, converging pointwise to $1$, such that 
$$
\mathcal{M}_{p, \sigma, \alpha} (e_N) \rightarrow 0 \quad \text{as} \quad N \rightarrow \infty.
$$
\end{lemma}

\begin{proof}
Let
\[
\rho(n):=(1\vee |n|)^{-2},
\]
and define
\[
X
:=
\left\{
u\in\ell^2(\Z,\rho):
\mathcal M_{p,\sigma,\alpha}(u)<\infty
\right\},
\]
equipped with the norm
\[
\|u\|_X^2
:=
\|u\|_{\ell^2(\Z,\rho)}^2
+
\mathcal M_{p,\sigma,\alpha}(u).
\]
We first show that $(X,\|\cdot\|_X)$ is a Hilbert space. It is enough
to prove completeness, the other properties being immediate.

Suppose that $(u_N)_{N\in\N}$ is a Cauchy sequence in $X$. Then
$(u_N)$ is Cauchy in $\ell^2(\Z,\rho)$, and hence converges strongly
to some $u\in\ell^2(\Z,\rho)$. In particular,
\[
u_N(x)\longrightarrow u(x),
\qquad x\in\Z.
\]
Observe that
\[
\sqrt{\mathcal M_{p,\sigma,\alpha}(u_N)}
=
\|U_N\|_{\ell^2(\Z^2,j)},
\]
where
\[
U_N(x,y):=u_N(x)-u_N(y)
\]
and
\[
j(x,y)
:=
 k_{-\alpha}(x)^{r_p}k_{-\alpha}(y)^{q_p}
    k_\sigma (x-y).
\]
Since $(u_N)$ is Cauchy in $X$, the sequence $(U_N)$ is Cauchy in
$\ell^2(\Z^2,j)$, and therefore converges strongly to some
$U\in\ell^2(\Z^2,j)$. For every $x\neq y$, this convergence implies
pointwise convergence, and hence
\[
U(x,y)
=
\lim_{N\to\infty}U_N(x,y)
=
u(x)-u(y).
\]
Thus
\[
\mathcal M_{p,\sigma,\alpha}(u)<\infty,
\]
and $u_N\to u$ in $X$. This proves that $X$ is complete.

Now by Banach–Alaoglu and Eberlein–Šmulian theorems there exists a subsequence $(\xi_{N_k})_{k \in \N}$ convergent weakly to some $\xi$ in $X$. This implies that $(\xi_{N_k})$ converges pointwise to $\xi$, and hence $\xi = 1$. By Mazur's lemma there exists a sequence $(e_k)_{k \in \N}$ strongly convergent to $1$, such that $e_k$ is a convex combination of $(\xi_{N_1},...,\xi_{N_k})$. Now we have
\begin{align*}
    \mathcal{M}_{p, \sigma, \alpha}(e_k) = \mathcal{M}_{p, \sigma, \alpha}(e_k - 1) \xrightarrow[]{k \rightarrow \infty}0.
\end{align*}
\end{proof}

Next, we construct an explicit sequence $\xi_N$ for which the form \eqref{M form} remains uniformly bounded.
\begin{lemma}\label{lem: elementary beta sum}
Let $r, s \in (0, 1)$ and $N \in \N$. Then 
\begin{equation}
\sum_{j=1}^N
\frac{1}{j^r(N+1-j)^s}
\lesssim
N^{1-r-s}.
\end{equation}
   
\end{lemma}

\begin{proof}
For $N=1$, the claim is immediate. Hence, assume $N\geq 2$. Let
\[
m:=\left\lceil \frac{N}{2}\right\rceil.
\]
Splitting the sum at $j=m$, we obtain
\[
\sum_{j=1}^N
\frac{1}{j^r(N+1-j)^s}
=
\sum_{j=1}^{m}
\frac{1}{j^r(N+1-j)^s}
+
\sum_{j=m+1}^{N}
\frac{1}{j^r(N+1-j)^s}.
\]
For $1\leq j\leq m$, we have $N+1-j\gtrsim N$, and hence
\[
\sum_{j=1}^{m}
\frac{1}{j^r(N+1-j)^s}
\lesssim
N^{-s}\sum_{j=1}^{m}j^{-r}
\lesssim
N^{-s}N^{1-r}
=
N^{1-r-s},
\]
where we used $r\in(0,1)$.

On the other hand, for $m<j\leq N$, we have $j\gtrsim N$. Therefore,
after the change of variables $k=N+1-j$,
\begin{align*}
\sum_{j=m+1}^{N}
\frac{1}{j^r(N+1-j)^s}
&\lesssim
N^{-r}\sum_{j=m+1}^{N}(N+1-j)^{-s}
\\
&=
N^{-r}\sum_{k=1}^{N-m}k^{-s}
\\
&\lesssim
N^{-r}N^{1-s}
=
N^{1-r-s},
\end{align*}
where we used $s\in(0,1)$.

Combining the two estimates proves the claim.
\end{proof}

\begin{lemma}\label{lem: boundedness of M form}
Let $N \in \N$ and define $\xi_N \in C_c(\Z)$ as 
\[
\xi_N(x)
:=
\begin{cases}
1,
    & |x|\leq N,\\[1mm]
2-\dfrac{|x|}{N},
    & N<|x|<2N,\\[2mm]
0,
    & |x|\geq2N.
\end{cases}
\]
Then $0 \leq \xi_N(x) \leq 1$ and converges pointwise to $1$. Moreover, if 
$$
\alpha \leq \frac{2\sigma+a_p}{2(a_p+1)},
$$
then
$$
\sup_{N \in \N} \mathcal{M}_{p, \sigma, \alpha}(\xi_N) < \infty.
$$    
\end{lemma}

\begin{proof}
It is immediate that $0 \leq \xi_N(x) \leq 1$ and that
$\xi_N(x) \to 1$ pointwise. We next estimate
$\mathcal M_{p,\sigma,\alpha}(\xi_N)$. 

The asymptotics of $k_\sigma$ in \eqref{e.ksigma-asymptotic}, the fact that $\xi_N$ is even, and the reverse triangle inequality inequality
$$
|x-y| \geq \big||x|-|y|\big|,
$$
give
$$
\mathcal{M}_{p, \sigma, \alpha}(\xi_N)  \lesssim \sum_{x, y \in \Z \atop |x| \neq |y|} \big(\xi_N(x) - \xi_N(y)\big)^2 k_{-\alpha}(x)^{r_p}k_{-\alpha}(y)^{q_p} \big||x|-|y|\big|^{-1-2\sigma}.
$$
Since $\xi_N$ and $k_{-\alpha}$ are even, decomposing the sum according to the signs of $x$ and $y$, and using the uniform bound for the terms involving the origin, we obtain
$$
\mathcal{M}_{p, \sigma, \alpha}(\xi_N) \lesssim  \sum_{x, y \in \N \atop x \neq y} \big(\xi_N(x) - \xi_N(y)\big)^2 k_{-\alpha}(x)^{r_p}k_{-\alpha}(y)^{q_p} |x-y|^{-1-2\sigma} + C,
$$
for some positive constant $C$, independent of $N$.

Using the asymptotics of $k_{-\alpha}$ in
\eqref{asymptotics of k}, together with
$\{r_p,q_p\}=\{1,a_p\}$ and the symmetry of the double sum under the interchange of $x$ and $y$, we obtain
\begin{equation}\label{positive M cutoff estimate}
\mathcal{M}_{p, \sigma, \alpha}(\xi_N) \lesssim  \sum_{x, y \in \N \atop x \neq y} \big(\xi_N(x) - \xi_N(y)\big)^2 \frac{1}{x^{1-2\alpha}} \frac{1}{y^{1-2\gamma_p}} |x-y|^{-1-2\sigma} + C,
\end{equation}
where 
$$
\gamma_p := \frac{1-a_p(1-2\alpha)}{2}.
$$
Hence, it is sufficient to bound 
\begin{equation}\label{positive M cutoff}
\sum_{\substack{x,y\in\N\\x\neq y}}
\big(\xi_N(x)-\xi_N(y)\big)^2
\frac{1}{x^{1-2\alpha}}
\frac{1}{y^{1-2\gamma_p}}
\frac{1}{|x-y|^{1+2\sigma}}.
\end{equation}
For convenience, define the parameters
\[
\beta:=1-2\alpha,
\qquad
\delta:=1-2\gamma_p.
\]
We split the sum in \eqref{positive M cutoff} into three regions,
according to the definition of $\xi_N$.

\medskip
\noindent\emph{Case 1: $1\leq y\leq N<x$.}
Since $|\xi_N(x)-\xi_N(y)|\leq1$, Lemma \ref{lem: elementary beta sum} yields
\begin{align*}
\sum_{1\leq y\leq N<x}
(\xi_N(x)-\xi_N(y))^2
\frac{1}{x^\beta y^\delta |x-y|^{1+2\sigma}}
&\lesssim
N^{-\beta}
\sum_{y=1}^N
\frac{1}{y^\delta}
\sum_{x>N}
\frac{1}{(x-y)^{1+2\sigma}}
\\
&\lesssim
N^{-\beta}
\sum_{y=1}^N
\frac{1}{y^\delta}
\frac{1}{(N+1-y)^{2\sigma}} \\
&\lesssim N^{-\beta+1-\delta-2\sigma} = N^{-1 +2(\alpha+\gamma_p-\sigma)}.
\end{align*}
Interchanging $x$ and $y$ gives the same estimate for the opposite
orientation.

\medskip
\noindent\emph{Case 2: $N<x,y<2N$.} In this region
\[
|\xi_N(x)-\xi_N(y)|
=
\frac{|x-y|}{N}.
\]
Therefore
\begin{align*}
\sum_{N<x,y<2N \atop x\neq y}
\big(\xi_N(x)-\xi_N(y)\big)^2
\frac{1}{x^\beta y^\delta |x-y|^{1+2\sigma}}
&= \frac1{N^2}
\sum_{\substack{N<x,y<2N\\x\neq y}}
\frac{|x-y|^{1-2\sigma}}{x^\beta y^\delta}\\
&   \leq N^{-2-\beta-\delta} \sum_{\substack{N<x,y<2N\\x\neq y}} |x-y|^{1-2\sigma}\\
& \lesssim N^{-1 + 2(\alpha+\gamma_p-\sigma)}.
\end{align*}

\medskip
\noindent\emph{Case 3: $N<y<2N\leq x$.} Again $|\xi_N(x)-\xi_N(y)|\leq1$, and arguing as above we get,
\begin{align*}
\sum_{N<y<2N\leq x}
(\xi_N(x)-\xi_N(y))^2
\frac{1}{x^\beta y^\delta |x-y|^{1+2\sigma}}
&\leq
N^{-\beta-\delta}
\sum_{N < y < 2N} 
\sum_{x \geq 2N} \frac{1}{|x-y|^{1+2\sigma}}\\
&\lesssim
N^{-\beta-\delta}
\sum_{N<y<2N}
\frac{1}{(2N-y)^{2\sigma}}\\
& \lesssim 
N^{-\beta-\delta}
\sum_{z=1}^N z^{-2\sigma}
\\
&\lesssim
N^{-\beta-\delta+1-2\sigma} =
N^{-1+2(\alpha+\gamma_p-\sigma)}.
\end{align*}
The opposite orientation is estimated in the same way.

The remaining regions give no contribution, since $\xi_N$ is constant
on $|x|\leq N$ and on $|x|\geq2N$. Combining the preceding estimates along with \eqref{positive M cutoff estimate},
we obtain
\[
\mathcal M_{p,\sigma,\alpha}(\xi_N)
\lesssim
1+N^{-1+2(\alpha+\gamma_p-\sigma)}.
\]
By the definition of $\gamma_p$,
\[
-1+2(\alpha+\gamma_p-\sigma)
=
2(a_p+1)\alpha-a_p-2\sigma.
\]
Hence, if
\[
\alpha\leq\frac{2\sigma+a_p}{2(a_p+1)},
\]
then
\[
-1+2(\alpha+\gamma_p-\sigma)\leq0,
\]
and therefore
\[
\sup_{N\in\N}
\mathcal M_{p,\sigma,\alpha}(\xi_N)<\infty.
\]
\end{proof}

We are now ready to prove the main result of this subsection.
\begin{proposition}[Criticality of the fractional Hardy weights]
\label{prop: criticality of fractional Hardy weights}
Let \(p\in(1,\infty)\), \(\sigma\in(0,1/2)\), and
\(\alpha\in(\sigma,1/2)\). Assume that
\[
    \alpha
    \leq
    \frac{2\sigma+a_p}{2(a_p+1)}.
\]
Then the Hardy weight \(w_{p,\sigma,\alpha}\) defined in
\eqref{fractional Hardy weight} is critical.
\end{proposition}

\begin{proof}
Fix $p \in (1, \infty), \, \sigma \in (0, 1/2)$, and
$$
\alpha \in \left(\sigma, \; \frac{2\sigma+a_p}{2(a_p+1)}\right].
$$
Suppose that $w$ is a Hardy weight for
\(\mathcal E_p^{\Z,\sigma}\) such that
\[
    w(x)\geq w_{p,\sigma,\alpha}(x),
    \qquad x\in\Z.
\]
Let
\[
h_{p, \sigma, \alpha}(x)
:=
k_{-\alpha}(x)^{q_p}.
\]
By the ground-state representation
\eqref{fractional ground state representation}, for every
$u\in C_c(\Z)$, we have
\begin{equation}\label{remainder estimate for fractional Laplacian}
0
\leq
\sum_{x\in\Z}
\bigl( w(x)-w_{p,\sigma,\alpha}(x)\bigr)|u(x)|^p
\leq
\mathcal R_{p,\sigma,\alpha}(u).    
\end{equation}
Formally,
\[
\mathcal R_{p,\sigma,\alpha}(h_{p, \sigma, \alpha})=0.
\]
However, $h_{p, \sigma, \alpha}$ is not finitely supported, so we approximate it by a suitable sequence of finitely supported functions. 

Lemma \ref{lem: boundedness of M form}, together with Lemma
\ref{lem: boundedness implies convergence to zero}, yields a finitely supported
sequence $0 \leq e_N \leq 1$ such that
\[
e_N(x) \to 1
\qquad \text{and} \qquad
\mathcal M_{p,\sigma,\alpha}(e_N) \to 0,
\]
as $N \to \infty$.

Define 
$$u_N (x) := h_{p, \sigma, \alpha}(x) e_N(x)^{\langle 2/p \rangle}.$$
Then $u_N \rightarrow h_{p, \sigma, \alpha}$ pointwise. By Lemma \ref{lem: from Fp to F2},
$$
\mathcal{R}_{p, \sigma, \alpha}(u_N) \leq 2 \mathcal{M}_{p, \sigma, \alpha}(e_N).
$$
Combining this with \eqref{remainder estimate for fractional Laplacian} and Fatou's lemma, we get
\begin{align*}
    0 \leq \sum_{x \in \Z} (w(x) - w_{p, \sigma, \alpha}(x))|h_{p, \sigma, \alpha}(x)|^p &\leq \liminf_{N \rightarrow \infty} \sum_{x \in \Z} (w(x) - w_{p, \sigma, \alpha}(x))  |u_N(x)|^p \\
    &\leq 2 \lim_{N \rightarrow \infty} \mathcal{M}_{p, \sigma, \alpha}(e_N) = 0.     
\end{align*}
Since $h_{p, \sigma, \alpha} > 0$, we conclude that
$$w (x) = w_{p,\sigma, \alpha}(x), \qquad x \in \Z.$$
This proves the criticality of $w_{p,\sigma, \alpha}$.
\end{proof}

\subsection{Proof of Theorem~\ref{thm: critical Hardy weights for fractional Laplacian}}
\label{subsec: proof of fractional theorem}

Let \(p\in(1,\infty)\), \(\sigma\in(0,1/2)\), and
\(\alpha\in(\sigma,1/2)\). Proposition~\ref{prop: fractional Hardy weights} gives the Hardy inequality
\[
    \mathcal E_p^{\Z,\sigma}(u)
    \geq
    \sum_{x\in\Z} w_{p,\sigma,\alpha}(x)|u(x)|^p,
\]
for \(u\in\ell^p(\Z,\mu_{p,\sigma,\alpha})\), with the weight
\(w_{p,\sigma,\alpha}\) given by \eqref{fractional Hardy weight}. It also gives
the strict positivity of \(w_{p,\sigma,\alpha}\).

By Proposition~\ref{prop: fractional reference measure and asymptotics},
\[
    \mu_{p,\sigma,\alpha}(x)\asymp (1+|x|)^{-2\sigma},
    \qquad x\in\Z.
\]
Hence the admissible space in the preceding inequality is precisely
\(\ell^p(\Z,(1+|x|)^{-2\sigma})\), up to equivalence of norms. The same
proposition gives the asymptotic formula
\[
    w_{p,\sigma,\alpha}(x)
    =
    \frac{\Psi(p,\sigma,\alpha)}{|x|^{2\sigma}}
    +
    O\left(
        \frac{1}
        {|x|^{2\sigma+1-\max\{2\sigma,1-2\alpha\}}}
    \right),
    \qquad |x|\to\infty.
\]

By Proposition~\ref{prop: criticality of fractional Hardy weights},
\(w_{p,\sigma,\alpha}\) is critical in the stated range. The proof is complete.
\qed

\begin{corollary}\label{cor: fractional contractivity}
Let \(p>1\), \(\sigma\in(0,1/2)\), and
\(\alpha\in(\sigma,1/2)\). Then the semigroup
\[
    \left(
    e^{\,t(-\Delta^\sigma+M_{w_{p,\sigma,\alpha}})}
    \right)_{t\geq0}
\]
is contractive on \(\ell^p(\Z)\). Moreover, \(w_{p,\sigma,\alpha}\) is
maximal, in the pointwise order, among nonnegative bounded potentials
with this property whenever
\[
    \alpha\leq\frac{2\sigma+a_p}{2(a_p+1)}.
\]
\end{corollary}

\begin{proof}
The kernel \(k_\sigma\) is uniformly summable and
\(w_{p,\sigma,\alpha}\in\ell^\infty(\Z)\). By
Theorem~\ref{thm: critical Hardy weights for fractional Laplacian},
\[
    \mathcal E_p^{\Z,\sigma}(u)
    \geq
    \sum_{x\in\Z}w_{p,\sigma,\alpha}(x)|u(x)|^p,
    \qquad u\in\ell^p(\Z),
\]
hence the contractivity follows from
Theorem~\ref{thm: contractivity of perturbed semigroup}.
If, in addition,
\[
    \alpha\leq\frac{2\sigma+a_p}{2(a_p+1)},
\]
and \(\widetilde w\geq w_{p,\sigma,\alpha}\) is a nonnegative bounded
potential for which
\(-\Delta^\sigma+M_{\widetilde w}\) generates a contraction semigroup
on \(\ell^p(\Z)\), then
Theorem~\ref{thm: contractivity of perturbed semigroup} implies that
\(\widetilde w\) is a Hardy weight. By the criticality of
\(w_{p,\sigma,\alpha}\), we have
\(\widetilde w=w_{p,\sigma,\alpha}\).
\end{proof}

%\bibliographystyle{plain}
%\bibliographystyle{abbrv}

%\bibliography{nat}

\end{document}